\documentclass{amsart}

\usepackage[margin=1.0in]{geometry} 
\usepackage{amsmath}
\usepackage{mathrsfs}
\usepackage{amsfonts}
\usepackage{dsfont}
\usepackage{amssymb}
\usepackage{graphicx}
\usepackage{accents}
\usepackage{bm}
\usepackage{caption}
\usepackage{tikz-cd}
\usetikzlibrary{positioning,shadings}
\usepackage{float}
\usepackage{natbib}
\usepackage{wasysym}
\usepackage{hyperref}
\usepackage{url}
\hypersetup{
    colorlinks=true,
    citecolor=,
    linkcolor=,
    filecolor=,
    urlcolor=blue
}
\setcitestyle{square,numbers}
\definecolor{OhioRed}{rgb}{0.73047, 0, 0} 
\definecolor{OhioGrey}{rgb}{0.8, 0.8, 0.8}

\theoremstyle{plain}
\newtheorem{theorem}{Theorem}
\newtheorem*{theorem*}{Theorem}

\newtheorem{corollary}[theorem]{Corollary}

\newtheorem*{corollary*}{Corollary}

\newtheorem{proposition}[theorem]{Proposition}
\newtheorem{conj}[theorem]{Conjecture}

\newtheorem{example}[theorem]{Example}

\theoremstyle{definition}
\newtheorem{definition}[theorem]{Definition}
\newtheorem{notation}[theorem]{Notation}

\theoremstyle{remark}
\newtheorem{remark}[theorem]{Remark}

\numberwithin{theorem}{section}
\numberwithin{figure}{section}
\numberwithin{equation}{section}

\newcommand{\C}{\mathbb{C}}

\newcommand{\hh}{\mathscr{H}}
\newcommand{\G}{\mathscr{G}}

\newcommand{\set}[1]{\left\{ #1 \right\}}
\newcommand{\ff}{\mathcal{F}}
\renewcommand{\gg}{\mathcal{G}}

\DeclareMathOperator{\ONE}{id}
\DeclareMathOperator{\ucat}{\mathbf{UCat}}
\DeclareMathOperator{\Tr}{Tr}
\DeclareMathOperator{\TLJ}{TLJ(\Gamma)}
\DeclareMathOperator{\hilb}{\mathbf{BigHilb}}

\renewcommand{\tilde}{\widetilde}
\renewcommand{\bar}{\overline}

\definecolor{OhioRed}{rgb}{0.73047, 0, 0} 
\definecolor{OhioGrey}{rgb}{0.8, 0.8, 0.8}

\makeatletter
    \def\@maketitle{%
  \newpage
  \null
  \vskip 2em%
  \begin{center}%
  \let \footnote \thanks
    {\LARGE \@title \par}%
    \vskip 1.5em%
    {\large
      \lineskip 1.5em%
      \bigskip
      \begin{tabular}[t]{c}%
        \Large{Giovanni Ferrer} \\ \small Department of Mathematics, University of Puerto Rico (Mayag\"uez) \\ \small Email: \texttt{giovanni.ferrer@upr.edu} \bigskip \bigskip \\
        
        \Large{Roberto Hern\'andez Palomares} \\ \small Department of Mathematics, The Ohio State University \\ \small Email: \texttt{hernandezpalomares.1@osu.edu}
      \end{tabular}\par}%
  \end{center}%
  \par
  \vskip 1.5em}
\makeatother

\title{\textsc{Classifying Module Categories for Generalized Temperley-Lieb-Jones $*$-2-Categories}}

\author{Giovanni Ferrer}

\author{Roberto Hernandez Palomares}

\date{November 2018}
\thanks{This research is supported by David Penneys' NSF CAREER grant 1654159.}

\begin{document}

\maketitle 

\begin{abstract}
Generalized Temperley-Lieb-Jones (TLJ) 2-categories associated to weighted bidirected graphs were introduced in unpublished work of Morrison and Walker. We introduce unitary modules for these generalized TLJ 2-categories as strong $*$-pseudofunctors into the $*$-2-category of row-finite separable bigraded Hilbert spaces. We classify these modules up to $*$-equivalence in terms of weighted bi-directed fair and balanced graphs in the spirit of Yamagami's classification of fiber functors on TLJ categories and DeCommer and Yamashita's classification of unitary modules for ${\rm Rep(SU}_q(2))$.
\end{abstract}
\makeatletter
\@setabstract
\makeatother

\section{Introduction}


The Temperley-Lieb-Jones (TLJ) algebras originate in Temperley and Lieb's article on ice-type lattices in statistical mechanics \cite{Temperley1971RelationsProblem}, and they were formalized by Jones in his study of finite index $\rm II_1$ subfactors \cite{Jones1983}.
Jones further used these algebras to define his famous knot polynomial using the Markov trace \cite{Jones1985AAlgebras}. Kauffman showed how to define the Jones polynomial via skein theory \cite{H.Kauffman1987StatePolynomial}, and it was later shown how to obtain the Jones polynomial from ${\rm TLJ}(\delta)$ viewed as a ribbon tensor category \cite{Reshetikhin1990RibbonGroups}.


A bridge between the TLJ categories and the representation categories of quantum groups are the so-called \textit{fiber functors}, which are strong monoidal functors $\rm{TLJ}(\delta)\rightarrow \rm{Vec}$, the category of finite dimensional vector spaces. 
In \cite{Yamagami}, Yamagami 
classified all fiber functors $\rm{TLJ}(\delta)\rightarrow \rm{Vec}$
using the spectra of certain associated (positive) linear maps. 

Now, each fiber functor $\rm{TLJ}(\delta)\rightarrow \rm{Vec}$ equips $\rm{Vec}$ with the structure of a module category for $\rm{TLJ}(\delta)$ \cite[\S7]{Etingof2015TensorCategories}. In fact, module categories for $\rm{TLJ}(\delta)$ were classified as generalized fiber functors into ${\rm BigVec}$, the rigid tensor category of bigraded vector spaces in terms of graphs with bilinear forms \cite{Etingof2004ModuleGraphs}.
In the unitary setting, DeCommer and Yamashita (\cite{DeCommer2013TannakaTheory} and \cite{DeCommer2013TannakaSU2}) classified module C*-categories for $\rm{SU}(2)$ which can be thought of as unitary fiber functors of the form $\ff: {\rm Rep(SU}_q(2))\rightarrow {\rm BigHilb}$, the rigid C*-tensor category (RC*TC) of bigraded Hilbert spaces in terms of \textit{fair and balanced weighted graphs}. (We refer the reader to Definition \ref{j k graded} for more details on bigraded Hilbert spaces, and we refer the reader to \cite{Jones2017} and to section 2.1 of \cite{Henriques2016CategorifiedCategories} for more details on RC*TC's.) Notice that
for an appropriate choice of $\delta$, TLJ$(\delta)$ is a RC*TC unitarily equivalent to ${\rm Rep(SU}_q(2))$.

In their preprint \cite{Morrison2010Graphpre-print}, Morrison and Walker introduce a generalized notion of the TLJ categories (see Definitions 3.1, 3.2, and 3.3 therein).
A \emph{bidirected weighted graph} consists of a countable locally finite directed graph $\Gamma$ together with a weight function $\delta: E(\Gamma)\rightarrow (0,\infty)$ and an involution of the edges denoted by $\overline{\, \cdot \,}$, which reverses the edges and satisfies $\delta(e) = \delta(\overline{\, e \,})$ for each edge $e\in E(\Gamma)$. Associated to a fixed bidirected weighted graph $(\Gamma, \delta, \overline{\, \cdot \,})$, we construct the $*$-2-category ${\rm TLJ}(\Gamma)$, where tensor product is determined by concatenation of paths in $\Gamma$
(see Definition \ref{TLJG} for more details). These TLJ $*$-2-categories generalize the ordinary TLJ categories; indeed, taking $\Gamma$ as follows recovers various TLJ RC*TC's:
\begin{itemize}
\item
a single vertex with a single self-dual loop recovers unshaded unoriented TLJ,
\item
a single vertex with two dual edges recovers unshaded oriented TLJ, and
\item
two vertices with two dual edges between them recovers 2-shaded TLJ.
\end{itemize}
We refer the reader to Example \ref{ejemplitos} for more details. 

In this article, we classify generalized fiber functors and module categories for the $*$-2-category ${\rm TLJ}(\Gamma)$ associated to a weighted bidirected graph $(\Gamma, \delta, \overline{\, \cdot \,})$.
That is, we classify $*$-\textit{pseudofunctors} \cite{NLab2018Pseudofunctor} into the $*$-2-category of separable/countably bigraded Hilbert spaces $\hilb$ (see Definition \ref{Jilberto}).
However, one quickly runs into difficulties arising from non-strictness of this 2-category, so we introduce the \emph{strict} $*$-2-category $\ucat$ of unitary countably semisimple categories with $*$-functors as 1-morphisms and uniformly bounded natural transformations as 2-morphisms (see Definition \ref{ucat}), which is $*$-2-equivalent to $\hilb$.
In this context, we work with strict $*$-pseudofunctors $\ff: \TLJ\rightarrow \ucat$, which are unambiguously determined by their action on the generators of ${\rm TLJ}(\Gamma)$,
whose images are called $q$-\textit{fundamental solutions} (to the conjugate equations) in \cite{DeCommer2013TannakaSU2}.
We refer the reader to Proposition \ref{canonical} for a rigorous statement.

To achieve a classification of these unitary modules, we first generalize the notion of a fair and balanced graph \cite{DeCommer2013TannakaSU2} to \textit{balanced $\Gamma$-fair graphs} \cite{Morrison2010Graphpre-print}, which can intuitively be thought of as $E(\Gamma)$-graded versions of ordinary fair and balanced graphs. 
\\

\noindent
\textbf{Definition \ref{fairgraph}:} We say a weighted directed graph $(\Lambda, w,\pi)$ with a graph homomorphism $\pi: \Lambda \rightarrow \Gamma$ is a $\Gamma$-\textbf{fair} graph if and only if $\pi$ is onto $V(\Gamma)$ and for each 
$e:a\rightarrow b$ in $E(\Gamma)$ and every vertex $\alpha \in \pi^{-1}(a)$
$$
\sum_{\substack{\{ \epsilon \, \mid \text{ source}(\epsilon) = \alpha\\ \text{ and }  \pi(\epsilon) = e\}}} w(\epsilon) = \delta_e.
$$
A remarkable example in this definition occurrs when the edge weighting comes from a vertex weighting $d:V(\Lambda)\rightarrow (0,\infty)$ as a ratio $w(\alpha\rightarrow \beta) = d(\alpha)/d(\beta).$ (Compare with the discussion on the bottom of page 3 of \cite{Morrison2010Graphpre-print}.) In Section \ref{last_section}, we will more closely explore how this notion compares to $\Gamma$-fairness. There are moreover additional desirable properties one could ask of a $\Gamma$-fair graph such as the existence a \textit{balanced} involution:\\

\noindent
\textbf{Definition \ref{balanced}:} We say a $\Gamma$-fair graph $(\Lambda,w,\pi)$ is \textbf{balanced} if and only if there exists an involution (\,$\overline{\, \cdot \,}$\,) on $E(\Lambda)$ that switches sources and targets, such that for every $\epsilon \in E(\Lambda)$
\begin{align*}
    w(\epsilon)w(\overline{\epsilon}) &= 1\ \text{ and}\\
    \pi(\overline{\epsilon}) &= \overline{\pi(\epsilon)}.
\end{align*}
Notice as in \cite[p2 Remarks 1]{DeCommer2013TannakaSU2}
that the existence of such an involution is a property, and not additional structure.
\\

We are now equipped to introduce our main result.
\\

\noindent
\textbf{Theorem \ref{classification}:} \textit{Every balanced $\Gamma$-fair graph arises from a $\Gamma$-fundamental solution in $\hilb$. 
Furthermore, there is an equivalence between isomorphism classes of balanced $\Gamma$-fair graphs and unitary isomorphism classes of strong $*$-pseudofunctors $\TLJ \rightarrow \hilb$.}
\\

We recover Proposition 2.3 in \cite{DeCommer2013TannakaSU2} for $\rm {Rep(SU}_q(2))$ for $q < 0$, by taking $\Gamma$ to be a single vertex and self-dual loop, which recovers unshaded unoriented TLJ.
We give more details in Example \ref{ejemplitos}.



We now provide the reader with a brief outline of this article. 
In Section \ref{Background}, we establish the framework and rigorously define most of the basic notions we use. 
We formally introduce bidirected weighted graphs together with an explanation on how to construct our prototypical $*$-2-category ${\rm TLJ}(\Gamma)$. 
We then introduce abstract $*$-2-categories, allowing us to sketch the $*$-2-equivalence between $\ucat$ and $\hilb$.

In Section \ref{equivalences}, we investigate the unitary equivalence of strong $*$-pseudofunctors $\ff:\rm{TLJ(}\Gamma\rm{)\rightarrow \mathcal{C}}$, where $\mathcal{C}$ is a strict $*$-2-category (see Definition \ref{pseudofunctor equivalence}), which requires the language of 3-categories.  
    In the spirit of \cite[\S2]{Yamagami} and \cite[Def.~1.3]{DeCommer2013TannakaSU2}, we define $\Gamma$-fundamental solutions in $\mathcal{C}$ (Definition \ref{fundamentalsolution}) as a generalization of solutions to the conjugate equations, a.k.a.~the zig-zag equations. 
When $\mathcal{C}$ is strict, $\Gamma$-fundamental solutions determine a unique (strict) unitary module. 
In fact, every such strong $*$-pseudofunctor $\rm{TLJ(}\Gamma\rm{)}\rightarrow \mathcal{C}$
turns out to be unitarily equivalent to a strict one, as stated in Proposition \ref{strictify}. 
Due to this result, it suffices to classify strict $*$-pseudofunctors. 
By means of the $*$-2-equivalence $\hilb\simeq \ucat$, we translate our classification to the case where $\mathcal{C} = \hilb$ to understand unitary equivalence of $\Gamma$-fundamental solutions in $\hilb$. 
    We close this section by following the techniques introduced in \cite{DeCommer2013TannakaSU2}, translating unitary equivalence of strict $*$-pseudofunctors (or that of $\Gamma$-fundamental solutions in $\hilb$) in terms of conjugate anti-linear operators.

Finally, in Section \ref{last_section}, we prove our main classification theorem stated above. 
To do so, we construct a balanced $\Gamma$-fair graph from a $\Gamma$-fundamental solution $S = (V,E,C)$ in $\hilb$. This requires the spectral data arising from the anti-linear forms associated to the maps $\{C^e\}_{e\in E(\Gamma)}$. Conversely, we demonstrate how to construct a strong $*$-pseudofunctor $\ff:\TLJ\rightarrow\hilb$ from any given balanced $\Gamma$-fair graph. We finally prove that these processes are mutually inverse, therefore establishing the desired equivalence. 

In the way of proving this result, we address the question posed by Morrison and Walker \cite{Morrison2010Graphpre-print} with regards to weighted graphs obeying a Perron-Frobenius type condition. (See Remark \ref{MW condition}.) We do so by providing necessary and sufficient conditions for a balanced $\Gamma$-fair graph to be of the type considered by Morrison and Walker. We conclude this last section by suggesting a connection between \textbf{Corollary B} as found in \cite{Coles2018TheAlgebras} involving right pivotal cyclic TLJ$(d)$-modules and our own work in the scope of Morrison and Walker's.

\subsection{Acknowledgements}
We are extremely grateful to David Penneys and Corey Jones for their intense sharing of ideas and techniques, and for all the advising they put into the development of this project. This classification work would not have been possible had they not been as generous with their time and willingness to discuss most of the ideas and proofs contained in this article. We are also very grateful to both mathematics departments at The University of Puerto Rico at Mayag\"uez and The Ohio State University for providing means for this interdepartmental connection to happen. Finally, we would like to thank the National Science Foundation as we were fully financially supported by David Penneys' NSF CAREER grant 1654159 and by the SAMMS program, which is organized by The Ohio State University and the University of Puerto Rico, Mayag\"uez.
\vspace*{\fill}
\section{Background}\label{Background}
\subsection{Graph generated Temperley-Lieb Jones Categories}
\begin{notation}
For a graph $\Gamma$, we denote by $V(\Gamma)$ and $E(\Gamma)$ the vertex set and edge set of $\Gamma$, respectively. 
\end{notation}
\begin{definition}\label{bidirectedgraph}\cite{Morrison2010Graphpre-print}
A weighted bidirected graph $(\Gamma, \delta, \overline{\, \cdot \,} )$  is a countable locally finite directed graph together with a weight function $$\delta: E(\Gamma)\rightarrow (0,\infty)$$ and an
involution called duality given by the map $$\overline{\, \cdot \,}: E(\Gamma)\rightarrow E(\Gamma).$$ Duality reverses sources and targets and the weight function has the property that $\delta(e) = \delta(\overline{e})$. Note that an edge with the same source and target might be self-dual, as loops are allowed in $\Gamma$. For simplicity, we will denote $(\Gamma, \delta, \overline{\, \cdot \,} )$ by $\Gamma$ and $\delta(e)$ by $\delta_e$.
\end{definition}
\begin{example} \label{GammaEx}
Here, we present an example of a bidirected graph where the edges $d$ and $e$ are self-dual. 

\begin{figure}[ht]
\[
\begin{tikzpicture}[thick]
\node (a) at ( 2,0) [circle,draw=black!100,fill=black!20] {};
\node (b) at ( 1,0) [circle,draw=black!100,fill=black!60] {};
\node (c) at ( 0,0) [circle,draw=black!100] {};
\draw[->] (a) to[bend left] node[below] {$\delta_{\overline{c}}$} (b);
\draw[->] (b) to[bend left] node[above] {$\delta_c$} (a);
\draw[->] (c) to [in=150,out=120,loop] node[left] {$\delta_a$} (c);
\draw[->] (c) to [in=210,out=240,loop] node[left] {$\delta_{\overline{a}}$} (c);
\draw[->,blue!50] (a) to [in=30,out=60,loop] node[right] {$\delta_d$} (a);
\draw[->] (a) to [in=-30,out=-60,loop] node[right] {$\delta_e$} (a);
\draw[->] (c) to[bend left] node[above] {$\delta_b$} (b);
\draw[->] (b) to[bend left] node[below] {$\delta_{\overline{b}}$} (c);
\end{tikzpicture}
\]
\caption{Weighted bidirected graph $\Gamma$}
\end{figure}
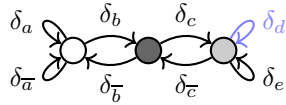

\end{example}
\begin{remark}
In the extent of this article, we only consider connected locally finite graphs.
\end{remark}

\begin{definition}
Let $\Gamma$ be a weighted bidirected graph, and let $a = (a_i)_{i=1}^n$ and $b = (b_j)_{j=1}^m$ be finite ordered sequences in $E(\Gamma)$ defining paths in $\Gamma$. This is, if $e,f$ are consecutive elements in either path, then the source of $f$ equals the target of $e$. Consider the unit square $[0,1]\times[0,1]$ with $n$ and $m$ points distinguished on the bottom and top ends, respectively. We correspond the $i^{th}$ bottom point from left to right with $a_i$ and the $j^{th}$ top point with $b_j$. A \textbf{$\Gamma$-Temperley-Lieb-Jones (TLJ$(\Gamma)$) diagram} from $a$ to $b$ consists of non-crossing smooth arcs starting from a point corresponding to an edge $e$ of $\Gamma$ and ending on either a point on the same unit square edge corresponding to $\overline{e}$, or on a point on the opposite unit square edge corresponding to $e$. 

\begin{itemize}
    \item A string in a diagram represents an edge of $\Gamma$ (from Example \ref{GammaEx}), where the shading of the region to the left of a string represents its source while the shading to the right represents its target. 
    
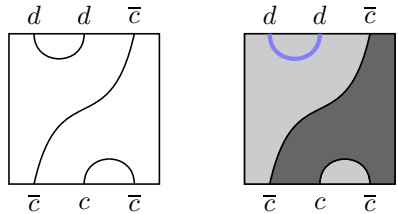
\begin{figure}[H]
    \centering
\begin{tikzpicture}[scale=1/6 pt, semithick] 
    \draw (-6,-6) -- (-6,6) -- (6,6) -- (6,-6) -- cycle; 
    \draw (0,-6) .. controls (0,-4.65) and (1,-4) .. (2,-4); 
    \draw (2,-4) .. controls (3,-4) and (4,-4.65) .. (4,-6); 
    \draw (-4,6) .. controls (-4,4.65) and (-3,4) .. (-2,4) 
.. controls (-1,4) and (0,4.65) .. (0,6);
    \draw (-4,-6) .. controls (-2,3) and (2,-3) .. (4,6); 
    \node at (-4,-7.5) {$\overline{c}$}; 
    \node at (0,-7.5) {$c$};
    \node at (4,-7.5) {$\overline{c}$};
    \node at (-4,7.5) {$d$}; 
    \node at (0,7.5) {$d$};
    \node at (4,7.5) {$\overline{c}$};  
\end{tikzpicture}
\hspace{25pt}
\begin{tikzpicture}[scale=1/6 pt, semithick] 
    \draw (-6,-6) -- (-6,6) -- (6,6) -- (6,-6) -- cycle; 
    \draw[fill=black!20] (0,-6) .. controls (0,-4.65) and (1,-4) .. (2,-4) .. controls (3,-4) and (4,-4.65) .. (4,-6) -- (0,-6); 
    \draw[fill=black!60] (-4,-6) .. controls (-2,3) and (2,-3) .. (4,6) -- (6,6) -- (6,-6) -- (4,-6) .. controls (4,-4.65) and (3,-4) ..  (2,-4) .. controls (1,-4) and (0,-4.65) .. (0,-6) -- (-4,-6); 
    \draw[fill=black!20] (-4,-6) .. controls (-2,3) and (2,-3) .. (4,6) -- (0,6) .. controls (0,4.65) and (-1,4) .. (-2,4) .. controls (-3,4) and (-4,4.65) .. (-4,6) -- (-6,6) -- (-6,-6) -- (-4,-6);
    \draw[fill=black!20] (-4,6) .. controls (-4,4.65) and (-3,4) .. (-2,4) 
.. controls (-1,4) and (0,4.65) .. (0,6) -- (-4,6);
    \draw[blue!50,ultra thick] (-4,6) .. controls (-4,4.65) and (-3,4) .. (-2,4) 
.. controls (-1,4) and (0,4.65) .. (0,6);
    \node at (-4,-7.5) {$\overline{c}$}; 
    \node at (0,-7.5) {$c$};
    \node at (4,-7.5) {$\overline{c}$};
    \node at (-4,7.5) {$d$}; 
    \node at (0,7.5) {$d$};
    \node at (4,7.5) {$\overline{c}$};      
\end{tikzpicture}
\end{figure}
\item By choosing one edge out of each duality pair whose source and target are the same, we assign orientations in order to distinguish the strings representing them.
    \begin{figure}[ht] 
\begin{tikzpicture}[scale=1/12 pt, thick]
    \draw[fill=black!20] (-12,-12) -- (-12,12) -- (12,12) -- (12, -12) -- cycle; 
    \draw[fill=black!60] (-8,-12) .. controls (-4,4) and (4,-4) .. (8,12) -- (12,12) -- (12,-12) -- cycle; 
\begin{scope}
    \clip (-12,-12) rectangle (12,12);
    \draw[fill=white] (2,-12) circle(6);
    \draw (2,-12) circle(3);
    \draw[blue!50,ultra thick] (-6,12) circle(3);
    \draw (2,12) circle(3);
\end{scope}
    \draw (-4,-12) -- (8,-12);
    \draw (1.5,-10) -- (2.5,-9) -- (1.5,-8);
    \draw (-12,-12) rectangle (12,12);
\end{tikzpicture}
\caption{Example of a $\TLJ$ diagram}
\end{figure}
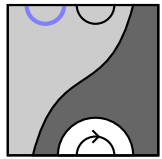
    \item Vertical and horizontal composition of $\TLJ$-diagrams are given by vertical stacking and horizontal juxtaposition, respectively. We remark that one can only vertically compose if the top and bottom ends of the given diagrams are labeled by the same path in $\Gamma$, and that horizontal composition is only possible whenever target of the last edge on the right-bottom (right-top) corner of the first diagram matches the source of the first edge on the left-bottom (left-top) corner of the second diagram.

\begin{equation}
    \begin{tikzpicture}[scale=1/14 pt, thick, baseline={([yshift=-\the\dimexpr\fontdimen22\textfont2\relax] current bounding box.center)}]
        \draw[fill=black!20] (-12,-12) -- (-12,12) -- (12,12) -- (12, -12) -- cycle; 
        \draw[fill=black!60] (-8,-12) .. controls (-4,4) and (4,-4) .. (8,12) -- (12,12) -- (12,-12) -- cycle; 
    \begin{scope}
        \clip (-12,-12) rectangle (12,12);
        \draw[fill=white] (2,-12) circle(6);
        \draw (2,-12) circle(3);
        \draw[blue!50,ultra thick] (-6,12) circle(3);
        \draw (2,12) circle(3);
    \end{scope}
        \draw (-4,-12) -- (8,-12);
        \draw (1.5,-10) -- (2.5,-9) -- (1.5,-8);
        \draw (-12,-12) rectangle (12,12);
    \end{tikzpicture}
    \, \, \, \, 
    \begin{tikzpicture}[scale=0.15, every node/.style={scale=0.75}] 
        \draw circle(0.7);
    \end{tikzpicture} \, \, \, \, 
    \begin{tikzpicture}[scale=1/14 pt, thick, baseline={([yshift=-\the\dimexpr\fontdimen22\textfont2\relax]
                        current bounding box.center)}]
        \draw[fill=black!20] (-12,-12) -- (-12,12) -- (12,12) -- (12, -12) -- cycle; 
        \draw[fill=black!60] (-8,12) .. controls (-4,-4) and (4,4) .. (8,-12) -- (12,-12) -- (12,12) -- cycle; 
    \begin{scope}
        \clip (-12,-12) rectangle (12,12);
        \draw[fill=white] (2,12) circle(6);
        \draw (2,12) circle(3);
        \draw[blue!50,ultra thick] (-6,-12) circle(3);
        \draw (2,-12) circle(3);
    \end{scope}
        \draw (-4,12) -- (8,12);
        \draw (2.5,10) -- (1.5,9) -- (2.5,8);
        \draw (-12,-12) rectangle (12,12);
    \end{tikzpicture} 
     \, \, 
    =
    \, \, 
    \begin{tikzpicture}[scale=1/14 pt, thick, baseline={([yshift=-\the\dimexpr\fontdimen22\textfont2\relax] current bounding box.center)}]
        \draw[fill=black!20] (-12,-12) -- (-12,12) -- (12,12) -- (12, -12) -- cycle; 
        \draw[fill=black!60] (8,-12) .. controls (4,-4) and (-8,-8) .. (-8,0) .. controls (-8,8) and (4,4) .. (8,12) -- (12,12) -- (12,-12) -- cycle; 
    \begin{scope} 
        \clip (-12,-12) rectangle (12,12);
        \draw[fill=white] (2,0) circle(5);
        \draw (2,0) circle(2.5);
        \draw[blue!50,ultra thick] (-6,-12) circle(3);
        \draw (2,-12) circle(3);
        \draw[blue!50,ultra thick] (-6,12) circle(3);
        \draw (2,12) circle(3);
    \end{scope}
        \draw[dashed] (-12,0) -- (12,0);
        \draw (1.5,3.5) -- (2.5,2.5) -- (1.5,1.5); 
        \draw (2.5,-3.5) -- (1.5,-2.5) -- (2.5,-1.5); 
        \draw (-12,-12) rectangle (12,12);
    \end{tikzpicture}
\end{equation}

\begin{equation}
    \begin{tikzpicture}[scale=1/14 pt, thick, baseline={([yshift=-\the\dimexpr\fontdimen22\textfont2\relax] current bounding box.center)}]
        \draw[fill=black!20] (-12,-12) -- (-12,12) -- (12,12) -- (12, -12) -- cycle; 
        \draw[fill=black!60] (-8,-12) .. controls (-4,4) and (4,-4) .. (8,12) -- (12,12) -- (12,-12) -- cycle; 
    \begin{scope}
        \clip (-12,-12) rectangle (12,12);
        \draw[fill=white] (2,-12) circle(6);
        \draw (2,-12) circle(3);
        \draw[blue!50,ultra thick] (-6,12) circle(3);
        \draw (2,12) circle(3);
    \end{scope}
        \draw (-4,-12) -- (8,-12);
        \draw (1.5,-10) -- (2.5,-9) -- (1.5,-8);
        \draw (-12,-12) rectangle (12,12);
    \end{tikzpicture} \, \,
    \otimes \, \,
    \begin{tikzpicture}[scale=1/14 pt, thick, baseline={([yshift=-\the\dimexpr\fontdimen22\textfont2\relax] current bounding box.center)}]
        \draw[fill=black!20] (-12,-12) -- (-12,12) -- (12,12) -- (12, -12) -- cycle; 
        \draw[fill=black!60] (8,12) .. controls (4,-4) and (-4,4) .. (-8,-12) -- (-12,-12) -- (-12,12) -- cycle; 
    \begin{scope}
        \clip (-12,-12) rectangle (12,12);
        \draw[fill=white] (-2,12) circle(6);
        \draw (-2,12) circle(3);
        \draw[blue!50,ultra thick] (6,-12) circle(3);
        \draw (-2,-12) circle(3);
    \end{scope}
        \draw (4,12) -- (-8,12);
        \draw (-1.5,10) -- (-2.5,9) -- (-1.5,8);
        \draw (-12,-12) rectangle (12,12);
    \end{tikzpicture} \, \, 
    =
    \, \, 
    \begin{tikzpicture}[scale=1/14 pt, thick, baseline={([yshift=-\the\dimexpr\fontdimen22\textfont2\relax] current bounding box.center)}]
        \draw[fill=black!20] (-12,-12) -- (-12,12) -- (12,12) -- (12, -12) -- cycle; 
        \draw[fill=black!60] (-10,-12) .. controls (-8,4) and (-4,-4) .. (-2,12) -- (10,12) .. controls  (8,-4) and (4,4) .. (2,-12) -- cycle; 
    \begin{scope}
        \clip (-12,-12) rectangle (12,12);
        \draw[fill=white] (-4.5,-12) ellipse (3.5 and 7);
        \draw[fill=white] (4.5,12) ellipse (3.5 and 7);
        \draw (-4.5,-12) ellipse (1.5 and 4.5);
        \draw (4.5,12) ellipse (1.5 and 4.5);
        \draw[blue!50,ultra thick] (9,-12) ellipse (1.5 and 4);
        \draw[blue!50,ultra thick] (-9,12) ellipse (1.5 and 4);    
        \draw (5,-12) ellipse (1.5 and 4);
        \draw (-5,12) ellipse (1.5 and 4);
    \end{scope}
        \draw (4,-12) -- (-8,-12);
        \draw (-4.7,-8.3) -- (-4.25,-7.5) -- (-5,-7);
        \draw (4.7,8.3) -- (4.25,7.5) -- (5,7);
        \draw[dashed] (0,-12) -- (0,12);
        \draw (-12,-12) rectangle (12,12);
    \end{tikzpicture}
\end{equation}

    \item An involution of $\TLJ$-diagrams is given by reflecting around an horizontal axis and reversing any string orientations.
\begin{equation}
    \left( \, \,
    \begin{tikzpicture}[scale=1/14 pt, thick, baseline={([yshift=-\the\dimexpr\fontdimen22\textfont2\relax] current bounding box.center)}]
        \draw[fill=black!20] (-12,-12) -- (-12,12) -- (12,12) -- (12, -12) -- cycle; 
        \draw[fill=black!60] (-8,-12) .. controls (-4,4) and (4,-4) .. (8,12) -- (12,12) -- (12,-12) -- cycle; 
    \begin{scope}
        \clip (-12,-12) rectangle (12,12);
        \draw[fill=white] (2,-12) circle(6);
        \draw (2,-12) circle(3);
        \draw[blue!50,ultra thick] (-6,12) circle(3);
        \draw (2,12) circle(3);
    \end{scope}
        \draw (-4,-12) -- (8,-12);
        \draw (1.5,-10) -- (2.5,-9) -- (1.5,-8);
        \draw (-12,-12) rectangle (12,12);
    \end{tikzpicture}
    \, \, \right)^{*}
    = \, \, \, \, \,
    \begin{tikzpicture}[scale=1/14 pt, thick, baseline={([yshift=-\the\dimexpr\fontdimen22\textfont2\relax]
                        current bounding box.center)}]
        \draw[fill=black!20] (-12,-12) -- (-12,12) -- (12,12) -- (12, -12) -- cycle; 
        \draw[fill=black!60] (-8,12) .. controls (-4,-4) and (4,4) .. (8,-12) -- (12,-12) -- (12,12) -- cycle; 
    \begin{scope}
        \clip (-12,-12) rectangle (12,12);
        \draw[fill=white] (2,12) circle(6);
        \draw (2,12) circle(3);
        \draw[blue!50,ultra thick] (-6,-12) circle(3);
        \draw (2,-12) circle(3);
    \end{scope}
        \draw (-4,12) -- (8,12);
        \draw (2.5,10) -- (1.5,9) -- (2.5,8);
        \draw (-12,-12) rectangle (12,12);
    \end{tikzpicture} 
\end{equation}

\end{itemize}
\begin{remark}
Similar to the standard Kauffman-diagrams, these TLJ$(\Gamma)$-diagrams are generated by families of cups and caps through vertical and horizontal composition.
\end{remark}

\begin{remark}
Through these graph-generated categories, one can obtain any simple Temperley-Lieb-like diagrams, with any number of string shadings, orientation, and region shadings. 
\end{remark}
\end{definition}
\begin{definition}\label{TLJG}  \cite{Morrison2010Graphpre-print} Let $\Gamma$ be a weighted bidirected graph. We define $\TLJ$, the \textbf{Temperley-Lieb Jones Category generated by} $\Gamma$, as the $*$-2-category defined as follows:
\begin{itemize}
    \item  Objects are vertices of $\Gamma$
    \item $1$-Morphisms are paths on $\Gamma$. In particular, for $a,b\in V(\Gamma),$ $$\hom(a,b):= \set{\mbox{paths in } \Gamma \mbox{ starting at } a \mbox{ and ending at } b}.$$ 
        Namely, the previously defined objects together with this collections of 1-morphisms make up the free category generated by $\Gamma.$ Notice that 1-composition becomes concatenation of paths, whenever the endpoint of the first path equals the starting point of the second, and is undefined otherwise.

    \item $2$-Morphisms from path $a$ to path $b$ are formal $\C$-linear combinations of simple $\TLJ$-diagrams from $a$ to $b$, modulo the $\delta_e$-equivalence relation, which trades closed $e$-loops and $\overline{e}$-loops for the scalar $\delta_e$.
    \begin{equation}
        \begin{tikzpicture}[scale=1/12 pt, thick, baseline={([yshift=-\the\dimexpr\fontdimen22\textfont2\relax] current bounding box.center)}]
            \draw[fill=black!20] (-12,-12) -- (-12,12) -- (12,12) -- (12, -12) -- cycle; 
            \draw[fill=black!60] (8,-12) .. controls (4,-4) and (-8,-8) .. (-8,0) .. controls (-8,8) and (4,4) .. (8,12) -- (12,12) -- (12,-12) -- cycle; 
        \begin{scope} 
            \clip (-12,-12) rectangle (12,12);
            \draw[fill=white] (2,0) circle(5);
            \draw (2,0) circle(2.5);
            \draw[blue!50,ultra thick] (-6,-12) circle(3);
            \draw (2,-12) circle(3);
            \draw[blue!50,ultra thick] (-6,12) circle(3);
            \draw (2,12) circle(3);
        \end{scope}
            \draw (1.5,3.5) -- (2.5,2.5) -- (1.5,1.5); 
            \draw (2.5,-3.5) -- (1.5,-2.5) -- (2.5,-1.5); 
            \draw (-12,-12) rectangle (12,12);
        \end{tikzpicture}
        \, \, \, \,\,
        = \, \, 
        \begin{tikzpicture}[scale=1/12 pt, thick, baseline={([yshift=-\the\dimexpr\fontdimen22\textfont2\relax] current bounding box.center)}]
            \draw[fill=black!20] (-12,-12) -- (-12,12) -- (12,12) -- (12, -12) -- cycle; 
            \draw[fill=black!60] (8,-12) .. controls (4,-4) and (-8,-8) .. (-8,0) .. controls (-8,8) and (4,4) .. (8,12) -- (12,12) -- (12,-12) -- cycle; 
        \begin{scope} 
            \clip (-12,-12) rectangle (12,12);
            \draw[fill=white] (2,0) circle(5);
            \draw[blue!50,ultra thick] (-6,-12) circle(3);
            \draw[blue!50,ultra thick] (-6,12) circle(3);
            \draw (2,12) circle(3);
            \draw (2,-12) circle(3);
        \end{scope}
            \node at (-17.5,0) {$\delta_a$};
            \draw (-12,-12) rectangle (12,12);
        \end{tikzpicture}
        \, \, \, \,\, = \, \, 
        \begin{tikzpicture}[scale=1/12 pt, thick, baseline={([yshift=-\the\dimexpr\fontdimen22\textfont2\relax] current bounding box.center)}]
            \node at (-20,0) {$\delta_a \delta_b$};
            \draw[fill=black!20] (-12,-12) -- (-12,12) -- (12,12) -- (12, -12) -- cycle; 
            \draw[fill=black!60] (8,-12) -- (8,12) -- (12,12) -- (12,-12) -- cycle; 
        \begin{scope} 
            \clip (-12,-12) rectangle (12,12);
            \draw[blue!50,ultra thick] (-6,-12) circle(3);
            \draw[blue!50,ultra thick] (-6,12) circle(3);
            \draw (2,12) circle(3);
            \draw (2,-12) circle(3);
        \end{scope}
            \draw (-12,-12) rectangle (12,12);
        \end{tikzpicture}
    \end{equation}
    \item Furthermore, we define horizontal and vertical composition of 2-morphisms as the linear extension of horizontal and vertical stacking of $\TLJ$-diagrams respectively, and we define an involution $*$ on the 2-morphisms of $\TLJ$ as the anti-linear extension of the involution of $\TLJ$-diagrams.
\end{itemize}
\end{definition}

\begin{remark}
We warn the reader that this is simply a $*$-2-category, as opposed to a 2-category with further analytic properties, such as being C*/W*.
\end{remark}

\begin{example}\label{ejemplitos}
The standard Temperley-Lieb Jones categories TLJ$(\delta)$, Temperley-Lieb Jones categories with oriented strings, and shaded TLJ are generated by the following weighted bidirected graphs $\Gamma_0$, $\Gamma_1$, $\Gamma_2$, and $\Gamma_3$ respectively:
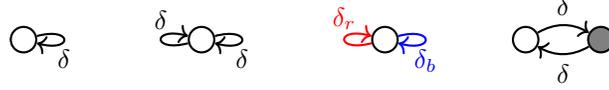
\begin{equation*}
\begin{tikzpicture}[thick, baseline={([yshift=-\the\dimexpr\fontdimen22\textfont2\relax] current bounding box.center)}]
\node (a) at ( 0,0) [circle,draw=black!100,fill=white] {};
\draw[->,black](a) to [loop right] node[below] {$\delta$} (a);
\draw[->,white](a) to [loop left] node[above] {$\delta$} (a);
\end{tikzpicture}
\hspace{25pt}
\begin{tikzpicture}[thick, baseline={([yshift=-\the\dimexpr\fontdimen22\textfont2\relax] current bounding box.center)}]
\node (a) at ( 0,0) [circle,draw=black!100,fill=white] {};
\draw[->](a) to [loop right] node[below] {$\delta$} (a);
\draw[->](a) to [loop left] node[above] {$\delta$} (a);
\end{tikzpicture}
\hspace{25pt}
\begin{tikzpicture}[thick, baseline={([yshift=-\the\dimexpr\fontdimen22\textfont2\relax] current bounding box.center)}]
\node (a) at ( 0,0) [circle,draw=black!100,fill=white] {};
\draw[->,blue](a) to [loop right] node[below] {$\delta_{b}$} (a);
\draw[->,red](a) to [loop left] node[above] {$\delta_{r}$} (a);
\end{tikzpicture}
\hspace{25pt}
\begin{tikzpicture}[thick, baseline={([yshift=-\the\dimexpr\fontdimen22\textfont2\relax] current bounding box.center)}]
\node (a) at ( 0,0) [circle,draw=black!100,fill=white] {};
\node (b) at ( 1,0) [circle,draw=black!100,fill=black!50] {};
\draw[->] (a) to [bend left] node[above] {$\delta$} (b);
\draw[->] (b) to [bend left] node[below] {$\delta$} (a);
\end{tikzpicture}
\end{equation*}

\noindent We include examples of $\TLJ$-diagrams corresponding to each of the previous graphs:

\begin{equation*}
\begin{tikzpicture}[scale=1/6 pt, thick, baseline={([yshift=-\the\dimexpr\fontdimen22\textfont2\relax] current bounding box.center)}] 
    \draw (-6,-6) -- (-6,6) -- (6,6) -- (6,-6) -- cycle; 
    \draw (0,-6) .. controls (0,-4.65) and (1,-4) .. (2,-4) .. controls (3,-4) and (4,-4.65) .. (4,-6); 
    \draw (-4,6) .. controls (-4,4.65) and (-3,4) .. (-2,4) 
.. controls (-1,4) and (0,4.65) .. (0,6);
    \draw (-4,-6) .. controls (-2,3) and (2,-3) .. (4,6); 
\end{tikzpicture}
\hspace{25pt}
\begin{tikzpicture}[scale=1/6 pt, thick, baseline={([yshift=-\the\dimexpr\fontdimen22\textfont2\relax] current bounding box.center)}] 
    \draw (-6,-6) -- (-6,6) -- (6,6) -- (6,-6) -- cycle; 
    \draw (0,-6) .. controls (0,-4.65) and (1,-4) .. (2,-4); 
    \draw (2,-4) .. controls (3,-4) and (4,-4.65) .. (4,-6); 
    \draw (-4,6) .. controls (-4,4.65) and (-3,4) .. (-2,4) 
.. controls (-1,4) and (0,4.65) .. (0,6);
    \draw (-4,-6) .. controls (-2,3) and (2,-3) .. (4,6); 
    \draw (-2.25,3.5) -- (-1.75,4) -- (-2.25,4.5); 
    \draw (-0.7,0.2) -- (0,0) -- (-.2,-0.7); 
    \draw (2.25,-3.5) -- (1.75,-4) -- (2.25,-4.5); 
\end{tikzpicture}
\hspace{25pt}
\begin{tikzpicture}[scale=1/6 pt, thick, baseline={([yshift=-\the\dimexpr\fontdimen22\textfont2\relax] current bounding box.center)}] 
    \draw (-6,-6) -- (-6,6) -- (6,6) -- (6,-6) -- cycle; 
    \draw[blue] (0,-6) .. controls (0,-4.65) and (1,-4) .. (2,-4) .. controls (3,-4) and (4,-4.65) .. (4,-6); 
    \draw[red] (-4,6) .. controls (-4,4.65) and (-3,4) .. (-2,4) 
.. controls (-1,4) and (0,4.65) .. (0,6);
    \draw[red] (-4,-6) .. controls (-2,3) and (2,-3) .. (4,6); 
\end{tikzpicture}
\hspace{25pt}
\begin{tikzpicture}[scale=1/6 pt, semithick, baseline={([yshift=-\the\dimexpr\fontdimen22\textfont2\relax] current bounding box.center)}] 
    \draw (-6,-6) -- (-6,6) -- (6,6) -- (6,-6) -- cycle; 
    \draw[fill=black!50] (0,-6) .. controls (0,-4.65) and (1,-4) .. (2,-4) .. controls (3,-4) and (4,-4.65) .. (4,-6) -- (0,-6); 
    \draw[fill=white] (-4,-6) .. controls (-2,3) and (2,-3) .. (4,6) -- (6,6) -- (6,-6) -- (4,-6) .. controls (4,-4.65) and (3,-4) ..  (2,-4) .. controls (1,-4) and (0,-4.65) .. (0,-6) -- (-4,-6); 
    \draw[fill=black!50] (-4,-6) .. controls (-2,3) and (2,-3) .. (4,6) -- (0,6) .. controls (0,4.65) and (-1,4) .. (-2,4) .. controls (-3,4) and (-4,4.65) .. (-4,6) -- (-6,6) -- (-6,-6) -- (-4,-6);
    \draw[fill=white] (-4,6) .. controls (-4,4.65) and (-3,4) .. (-2,4) 
.. controls (-1,4) and (0,4.65) .. (0,6) -- (-4,6);
    \draw (-4,6) .. controls (-4,4.65) and (-3,4) .. (-2,4) 
.. controls (-1,4) and (0,4.65) .. (0,6);
\end{tikzpicture}
\end{equation*}

\end{example}
\medskip
\subsection{Unitary Modules for TLJ($\Gamma$)}

\begin{notation}
In this paper, we use the terms bicategories and 2-categories indistinguishably. So we make no general assumptions on the strictness of our 2-categories. We will make our assumption of strictness explicit whenever required. We denote the composition of 1-morphisms by $\otimes$ and also for horizontal composition of 2-morphisms. We denote the vertical composition of 2-morphism with $\circ$.
\end{notation}

\begin{definition}
A $*$-2-category is a dagger-category enriched category. Namely, for arbitrary 1-morphisms $A,B,C$ and each pair $\eta\in 2\hom(A\Rightarrow B)$, $\kappa\in 2\hom(B\Rightarrow C)$ we have that $(\kappa\circ\eta)^* = \eta^*\circ\kappa^*$, and $(\kappa\otimes\eta)^* = \kappa^*\otimes\eta^*,$ whenever the 2-morphisms are composable. (See section 2 in \cite{Jones2017} for a more detailed discussion on dagger/ $C^*$-categories.)
\end{definition}

\begin{definition}
Given 2-categories $\mathcal{C}$ and $\mathcal{D}$, a pseudofunctor $\ff: \mathcal{C} \rightarrow \mathcal{D}$ consists of a triplet $(\ff,\mu,\iota),$ defined as follows:
\begin{itemize}
    \item For each 0-morphism $x\in \mathcal{C}$, a 0-morphism $\ff(x)\in \mathcal{D}$;
    \item For each hom-category $C(x , y)$ in $\mathcal{C}$, a functor $\ff_{x,y}: \mathcal{C}(x, y) \rightarrow \mathcal{D}(\ff(x) , \ff(y))$;
    \item For each 0-morphism $x$ of $\mathcal{C}$, an invertible 2-morphism (or 2-isomorphism) $\iota_x: \ONE_{\ff(x)} \Rightarrow \ff_{x,x}(\ONE_x)$.
    \item \textit{The tensorator} is a natural isomorphism $\mu$ given by a collection of 2-isomorphisms of the form  $\mu_{\varphi,\psi}:\ff(\varphi)\otimes\ff(\psi) \Rightarrow \ff(\varphi\otimes\psi),$ where $\varphi,\psi$ are 1-morphisms in $\mathcal C$.
\end{itemize}
We limit ourselves to mention there are some coherence axioms involved, but we do not mention them here. Rather, we direct the interested reader to the description found in \textbf{nLab}.\cite{NLab2018Pseudofunctor}\\ Furthermore, if $\mathcal{C},\mathcal{D}$ are $*$-2-categories, for every 0-morphism $x\in \mathcal{C}$ we have that $\iota_x$ is unitary, for every pair of 1-morphisms $\varphi, \psi$ in $\mathcal{C}$ the tensorator $\mu_{\varphi,\psi}$ is unitary, and if $\ff(\varphi^*) = \ff(\varphi)^*$ holds in $\mathcal{C}$, we then say that $\ff$ is a $*$-pseudofunctor. We conclude this definition by reminding the reader that in a $*$-2-category, unitarity for a 2-morphism $u$ means that $u^* = u^{-1}.$
\end{definition}

\begin{notation}
We use the terms unitary categories and countably semisimple $C^*$-categories indistinguishably. For a detailed explanation on $C^*$-categories see \cite{Jones2017}.
\end{notation}

\begin{definition} \label{ucat}
Let $\ucat$ be the 2-category whose 0-morphisms consist of unitary categories, with $*$-functors as 1-morphisms, and (uniformly) bounded natural transformations as 2-morphisms. We turn $\ucat$ into a $*$-2-category as follows. Let $\alpha: \mathcal{F} \Rightarrow \mathcal{G}$ be a 2-morphism in $\ucat$, where $\mathcal{F},\mathcal{G}: \mathcal{C} \rightarrow \mathcal{D}$ are $*$-functors. Namely, $\alpha = \{\alpha_c:\ff(c)\rightarrow \gg(c)\}_{c \in \mathcal{C}}$ is a family of morphisms in $\mathcal{D}$ indexed by the objects of $\mathcal{C}$. We then define the involution $*$ in $\ucat$ as $\alpha^* := \{\alpha_c^{*_{\mathcal{D}}}:\gg(c)\rightarrow \ff(c)\}_{c\in \mathcal C},$ where $*_{\mathcal{D}}$ in the involution in the unitary category $\mathcal{D}$.
\end{definition}

\begin{remark}\label{UCat is strict}
Notice that $\ucat$ is strict, as tensoring 1-morphisms is given as composition of functors. We will also suppress all associators and unitors.
\end{remark}

\begin{notation}\label{hom notation}
For any given 2-category $\mathcal C,$ say TLJ$(\Gamma)$ or $\ucat$, we will write $f\in 1\mathcal C$ to simply mean $f$ is a 1-morphism between two objects in $\mathcal C,$ without necessarily specifying the objects. We do similarly for 2-morphisms.
\end{notation}

From this point on, we focus our attention into classifying those unitary $\TLJ$-modules which we present as strong $*$-2-functors $\ff:\TLJ \rightarrow \ucat$. In order to do so, we introduce an auxiliary $*$-2-category, which is $*$-2-equivalent to $\ucat$, allowing us to use linear-algebraic tools in the spirit of \cite{Yamagami} and \cite{DeCommer2013TannakaSU2}. 

\begin{definition} \label{j k graded} Let $J$ and $K$ be countable sets. We denote by $\text{Hilb}^{J \times K}_f$ the category of $J \times K$-graded Hilbert spaces 
$$\hh = \bigoplus_{\substack{v \in J \\ w \in K}} \hh_{vw}$$
\\
\noindent such that $\sup_v \sum_w \dim(\hh_{vw}) < \infty$. In other words, $\hh_{vw}$ is finite dimensional for every pair $v,w$ and only a finite number of them are non-trivial when fixing the first index. (One can think of ``row finite'' matrices of finite dimensional Hilbert Spaces.) The morphisms are then defined as bounded operators

$$f = \bigoplus_{\substack{v \in J \\ w \in K}} f_{vw}: \hh \rightarrow \G,$$
where $f_{vw}: \hh_{vw} \rightarrow \G_{vw}$ are morphisms in Hilb$_{f.d.}$. The composition of morphisms $f,g \in \text{Hilb}^{J \times K}_f$ is then given by entry-wise composition, namely
\begin{equation}\label{1mult}
    g \circ f := \bigoplus_{\substack{v \in J \\ w \in K}} g_{vw} \circ f_{vw}.
\end{equation}
\end{definition}

\begin{definition}\label{Jilberto}
We define $\hilb$ as the $*$-2-category of bigraded Hilbert spaces with countable sets as 0-morphisms and $\hom(J , K) = \text{Hilb}^{J \times K}_f$. The composition of 1-morphisms denoted by $\otimes$ for $\hh: J \rightarrow K$ and $\G: K \rightarrow L$ is defined as
$$\hh \otimes \G := \bigoplus_{\substack{v \in J \\ w \in L}} \bigoplus_{k \in K}  \hh_{vk} \otimes \gg_{kw}$$
\\
\noindent where the $\otimes$ on the right side is the tensor product of Hilbert spaces. This operation is analogous to matrix multiplication. Note that for each object $J$, the identity 1-morphism $\ONE_J$ is given by
$$\ONE_J = \bigoplus_{v,w \in J} \delta_{v,w} \cdot \C,$$
where $\delta_{v,w} := 1$ when $v = w$ and $\delta_{v,w} := 0$ otherwise. Recall that the composition of 2-morphisms was defined in Equation (\ref{1mult}). We turn $\hilb$ into a $*$-2-category as follows. For each 2-morphism $f = \bigoplus_{v,w} f_{vw}: \hh \rightarrow \G$ we define its adjoint $f^* := \bigoplus_{v,w} f^*_{vw}: \G \rightarrow \hh$, where $f^*_{vw}$ is the adjoint of $f_{vw}$ as a bounded linear operator.
\end{definition}

It is well known amongst experts that the 2-category of semi-simple linear categories is 2-equivalent to the 2-category of bigraded vector spaces. In a similar fashion, the W$^*$-2-category of countably semi-simple C$^*$-categories is $*$-2-equivalent to the W$^*$-2-category $\hilb$. We will only provide a proof sketch of the latter, as proving these statements here would take us too far afield. 


\begin{proposition} \label{catequivalence} 
There exists a unitary $*$-2-equivalence of 2-categories $\Theta^{-1}: \ucat \rightarrow\hilb$.
\end{proposition}
\begin{proof}[Sketch of proof]
For $\mathcal C\in$ $\ucat$, by countable semi-simplicity together with the axiom of choice, there exists a countable set $\mbox{Irr}(\mathcal C)$ defining a complete set of representatives of isomorphism classes of simple objects in $\mathcal C.$ Now, if $\mathcal C, \mathcal D\in \ucat$, for $\ff:\mathcal C\rightarrow \mathcal D,$ a 1-morphism in U-Cat, we produce the category $\text{Hilb}^{\mbox{Irr}(\mathcal C) \times \mbox{Irr}(\mathcal D)}_f$ as follows: for $e_j\in\mbox{Irr}(\mathcal{C}), f_i\in\mbox{Irr}(\mathcal{D}),$ we turn the $i,j$ vector space component $\hh_{i,j}:= \mathcal{D}(f_i\rightarrow \ff(e_j))$ into a Hilbert space by means of the sesqui-linear form $\langle \varphi,\psi \rangle := \psi^*\circ\varphi \in \mbox{End}(f_i) = \mathbb C,$ which defines an inner product. Finally, if $\eta:\ff \Rightarrow \mathcal G$ is a 2-morphism in $\ucat$ presented as a family of functions $\eta = \{\eta_a\}_{a\in \mathcal C},$ we construct a 2-morphism in $\hilb$ using the expression $\eta_{{e_j}_i}\circ - := (\mathcal D(f_i, \ff(e_j))_{i,j}\Rightarrow (\mathcal D (f_i, \mathcal G(e_j))_{i,j}$. We spare the remaining details on why the structure maps are unitary and $\Theta^{-1}$ is essentially surjective on objects and 
fully faithful, thus being part of an equivalence of 2-categories.
\end{proof}


\vspace*{\fill}
\section{Equivalences of Pseudofunctors}\label{equivalences}
In this section, we develop the tools necessary to classify unitary modules for $\TLJ$. Afterwards we rephrase the equivalence of unitary modules in terms of certain anti-linear operators and state some useful properties.

\begin{definition}\label{fundamentalsolution}

Let $\mathcal{C}$ be a $*$-2-category which need not be strict. We define a $\Gamma$-fundamental solution in $\mathcal C$ as a triplet $S = (V,E,C)$ given as follows: $V = \set{V^a}_{a \in V(\Gamma)}$ are 0-morphisms in $\mathcal{C}$ indexed by the vertices of $\Gamma$, $E = \set{E^e}_{e \in E(\Gamma)}$ are 1-morphisms in $\mathcal{C}$ indexed by the edges of $\Gamma$, and $C = \set{C^e}_{e \in E(\Gamma)}$ are 2-morphisms in $\mathcal{C}$, where $C^e: \ONE_{V^a} \Rightarrow E^e \otimes E^{\overline{e}}$ satisfies the following zigzag relations for every $e \in E(\Gamma)$ 
\begin{enumerate}
    \item $\big((C^e)^* \otimes \ONE_{E^e} \big) \circ \big(\ONE_{E^e} \otimes C^{\overline{e}}\big) = \ONE_{E^e}$ and
    \item $(C^e)^* \circ C^e = \delta_e \cdot \ONE_{V^a}$
\end{enumerate}
This is represented diagrammatically as follows:

\begin{equation} \label{fig:zigzag2}
    \begin{tikzpicture}[scale=1/6 pt, semithick, baseline={([yshift=-\the\dimexpr\fontdimen22\textfont2\relax] current bounding box.center)}] 
    \tikzstyle{test}=[text=black]
        \begin{scope}
            \clip (-6,-5) rectangle (6,7);
            \fill[black!20] (-6,-5) rectangle (6,7);
            \fill[black!60] (-6,-5) -- (-4,-5) -- (-4,2) .. controls (-4,-4.65 + 6 + 2) and (-3,-4 + 6 + 2) .. (-2,4) .. controls (-1,4) and (0,-4.65 + 6 + 2) .. (0,2) -- (0,0) .. controls (0,4.65 - 6) and (1,4 - 6) .. (2,-2) .. controls (3,4 - 6) and (4,4.65 - 6) .. (4,0) -- (4,7) -- (-6,7); 
        \end{scope}
        \draw (-6,-5) rectangle (6,7);
        \draw (0,2) .. controls (0,-4.65 + 6 + 2) and (-1,4) .. (-2,4) 
    .. controls (-3,-4 + 6 + 2) and (-4,-4.65 + 6 + 2) .. (-4,-6 + 6 + 2);
        \draw (-2.25,-3.5 + 6 + 2) -- (-1.75, -4 + 6 + 2) -- (-2.25, -4.5 + 6 + 2); 
        \draw (4,6 - 6) .. controls (4,4.65 - 6) and (3,4 - 6) .. (2,4 - 6) 
    .. controls (1,4 - 6) and (0,4.65 - 6) .. (0,6 - 6);
        \draw (1.75,3.5 - 6) -- (2.25,4 - 6) -- (1.75,4.5 - 6); 
        \draw (-4,2) -- (-4, -5); 
        \draw (4,0) -- (4,7); 
        \draw (0,0) -- (0,2); 
        \draw (-3.5, -2) -- (-4,-1.5) -- (-4.5, -2); 
        \draw (3.5,3.25) -- (4, 3.75) -- (4.5,3.25); 
        \node at (-3.3,-4) [test] {$e$};
        \node at (-3.3,1) [test] {$e$};
        \node at (-0.7,1.2) [test] {$\bar{e}$};
        \node at (4.7,1) [test] {$e$};
        \node at (4.7,6) [test] {$e$};
    \end{tikzpicture}
    \,\,\, = \,\,\,
    \begin{tikzpicture}[scale=1/6 pt, semithick, baseline={([yshift=-\the\dimexpr\fontdimen22\textfont2\relax] current bounding box.center)}] 
       \begin{scope}
            \clip (-6,-2) rectangle (6,10);
            \fill[black!60] (-6,-2) rectangle (0,10);
            \fill[black!20] (0,-2) rectangle (6,10);
        \end{scope}
        \draw (-6,-2) rectangle (6,10);
        \draw (0,-2) -- (0,10); 
        \draw (-0.5,3.5) -- (0,4) -- (0.5,3.5); 
        \node at (0.7,-1) {$e$};
        \node at (0.7,9) {$e$};
    \end{tikzpicture}
    \hspace{75pt}
    \begin{tikzpicture}[scale=1/6 pt, semithick, baseline={([yshift=-\the\dimexpr\fontdimen22\textfont2\relax] current bounding box.center)}]
            \fill[black!60] (-8,-5) -- (-8,7) --(4,7) -- (4,-5) -- cycle;
            \draw (-8,-5) -- (-8,7) --(4,7) -- (4,-5) -- cycle;
        \filldraw[black!20] (-2,1) circle(3);
        \draw (-2,1) circle(3);
        \node at (0,-3.3) {};
        \node at (-0.4,1.2) {$\overline{e}$};
        \node at (-3.6,1) {$e$};
        \draw (-2.25,4.5) -- (-1.75,4) -- (-2.25,3.5); 
        \draw (-1.75,-2.5) -- (-2.25,-2) -- (-1.75,-1.5); 
    \end{tikzpicture} 
    \,\,\,\,\, =
    \begin{tikzpicture}[scale=1/6 pt, semithick, baseline={([yshift=-\the\dimexpr\fontdimen22\textfont2\relax] current bounding box.center)}]
        \node at (-10.5,0.5) {$\delta_e$};
        \fill[black!60] (-8,-5) -- (-8,7) --(4,7) -- (4,-5) -- cycle;
        \draw (-8,-5) -- (-8,7) --(4,7) -- (4,-5) -- cycle;
    \end{tikzpicture}
\end{equation}

\end{definition}
\medskip
\begin{example}\label{solinhilb}
Given a bidirected weighted graph $\Gamma$, we shall  describe $\Gamma$-fundamental solution in $\hilb,$ denoted by $(J,H,C)$. Here, $J = \set{J^a}_{a \in V(\Gamma)}$ is a family of (grading) sets indexing the vertices of $\Gamma$, $H = \set{\hh^e}_{e \in E(\Gamma)}$ is a family of bigraded Hilbert spaces graded by the edges of $\Gamma$ such that $\hh^e\in \text{Hilb}^{J^a \times J^b}_f$, and $C = \set{C^e}_{e \in E(\Gamma)}$ is a family of 2-morphisms in $\hilb$ such that $C^e = \bigoplus_v \sum_w C^e_{vw}.$ To further explain this notation trick, for each $(e: a\rightarrow b)\in E(\Gamma)$, for every fixed $v\in J^a,$ summing over $w\in J^b$ collects all the cups corresponding to the triple $(e,v,w)$ and, as $v$ ranks over the whole set $J^a,$ the direct sum places each corresponding combination of $C^e_{vw}$ cups into the appropriate diagonal slot. Here, the maps $C^e$ are collections of linear maps of the form
$$C^e_{vw}: \C \rightarrow \hh^e_{vw} \otimes \hh^{\overline{e}}_{wv},$$
where $v\in J^a$ and $w\in J^b.$ These form solutions to the equations (1) and (2) above, in the following fashion:
\begin{enumerate}
    \item $\big((C_{vw}^e)^* \otimes \ONE_{\hh_{vw}^e} \big) \circ \big(\ONE_{\hh_{vw}^e} \otimes C^{\overline{e}}_{wv}\big) = \ONE_{\hh_{vw}^e}$ and
    \item $(C_{vw}^e)^* C_{vw}^e = \delta_e\cdot 1_{vv}.$
\end{enumerate}

Here,  $1_{vv} \in \ONE_{v,v}^{J^a} = (\delta_{v,w}\cdot \C)_{v,v}$ is the complex number one.
\end{example}

\begin{proposition}\label{canonical}
A $\Gamma$-fundamental solution $S = (V,E,C)$ in a strict $*$-2-category $\mathcal{C}$ uniquely determines a canonical strict $*$-pseudofunctor $\ff_{S}: \TLJ \rightarrow \mathcal C$, such that $\ff_S(a) = V^a$ for every vertex $a$, $\ff_S(e) = E^e$ for every edge $e$, and $\ff_S (^e \cup ^{\overline{e}}) = C^e$ for every cup in $\TLJ$. 
\end{proposition}
\begin{proof} Let $S$ be as described above. At the level of 0-morphisms, $\ff_S$ has been completely described in the statement. For 1-morphisms, notice that for any path $f = (f_i)_i$ in $\Gamma$, we can unambiguously define $\ff_S(f) = \ff_S(\otimes_i f_i) := \otimes_i \ff(f_i)$, since $\mathcal C$ is strict. Finally, every 2-morphism in $\TLJ$ is generated by a sum of adjointing and composing cups $\set{^e\cup^{\overline{e}}}$ horizontally and vertically along with single strand and empty diagrams, so we shall now use this property to show how to define the action $\ff_S$ on 2-morphisms. We define $\ff_S$ for an empty diagram $id_{id_{a}}$ trivially as $\ff_S(id_{id_a}) := id_{\ff_S(id_a)} := id_{id_{\ff_S(a)}} = id_{id_{V^a}}$. Similarly, we let $\ff_S$ act on diagrams with a single strand $id_e$ by $\ff_S(id_e) := id_{\ff_S(e)} = id_{E^e}$. Let $g$ be an arbitrary path in $\Gamma$ and $\alpha\in 2\TLJ(f\Rightarrow g)$ be an arbitrary Kauffmann diagram. By the strictness of TLJ$(\Gamma)$, we can freely rearrange parenthesis in either $g$ or $f$. Let us choose an arrangement of parenthesis for both so that if $\alpha$ contains a cup or a cap, then both of the edges in $g$ or $f$ involved in the domain/codomain of any given cup/cap appear associated. Moreover, if $\alpha$  contains nested cups or caps, by isotopy, we can ``vertically separate" them by stacking enough Kauffmann diagrams consisting of horizontal composition of vertical strings and/or (colored) vacua between any two nested cups or caps. We denote each horizontal strip in the resulting diagram by $_k\alpha := \bigotimes_{i_k = 1}^{T_k} {\,}_k\alpha_{i_k}$. Here, each $_k\alpha_{i_k}$ must then either be an empty diagram, a string, a cup or a cap.

We therefore expressed our Kauffmann diagram as $\alpha = \circ_{k = 1}^L {\,}_k\alpha,$ dictating a decomposition of $\alpha$ as a grid consisting of (colored) empty diagrams, single cups, single caps and vertical strands, where at most one of each is found inside each square. This is progressively depicted in the following example:

\begin{equation}
    \begin{tikzpicture}[scale=1/8 pt, thick, baseline={([yshift=-\the\dimexpr\fontdimen22\textfont2\relax] current bounding box.center)}]
        \draw[fill=black!20] (-12,-12) -- (-12,12) -- (12,12) -- (12, -12) -- cycle; 
        \draw[fill=black!60] (-8,-12) .. controls (-4,4) and (4,-4) .. (10,12) -- (-12,12) -- (-12,-12) -- cycle;
        \draw[fill=white] (-10,-12) .. controls (-4,6) and (4,-2) .. (8,12) -- (-12,12) -- (-12, -12) -- cycle;
    \begin{scope} 
        \clip (-12,-12) rectangle (12,12);
        \draw[fill=black!60] (7,-12) circle(3); 
        \draw[fill=black!60] (-2,-12) circle(4);
        \draw[fill=white] (-2,-12) circle(2);
        \draw[fill=black!60] (-4,12) circle(4);
        \draw[fill=black!20] (-4,12) circle(2);
    \end{scope}
        \draw (-12,-12) -- (-12,12) -- (12,12) -- (12, -12) -- cycle;
    \end{tikzpicture}
    =
    \begin{tikzpicture}[scale=1/8 pt, thick, baseline={([yshift=-\the\dimexpr\fontdimen22\textfont2\relax] current bounding box.center)}]
        \draw[fill=black!20] (-12,-12) -- (-12,12) -- (12,12) -- (12, -12) -- cycle; 
        \draw[fill=black!60] (-2,-12) -- (-2,12) -- (-12,12) -- (-12,-12) -- cycle;
        \draw[fill=white] (-4,-12) -- (-4,12) -- (-12,12) -- (-12, -12) -- cycle;
        \draw[fill=black!60] (1,-5.9) -- (1,-12) -- (5,-12) -- (5,-5.9);
        \draw[fill=white] (2,-9.9) -- (2,-12) -- (4,-12) -- (4,-9.9);
        \draw[fill=black!60] (-11,6) -- (-11,12) -- (-7,12) -- (-7,6);
        \draw[fill=black!20] (-10,9.9) -- (-10,12) -- (-8,12) -- (-8,9.9);
    \begin{scope} 
        \clip (-12,-12) rectangle (12,12);
        \draw[fill=black!60] (9,-12) circle(2); 
    \end{scope}
    \begin{scope}
        \clip (0,-6) rectangle (6,0);
        \draw[fill=black!60] (3,-6) circle(2);
    \end{scope}
    \begin{scope}
        \clip (0,-10) rectangle (6,-6);
        \draw[fill=white] (3,-10) circle(1);
    \end{scope}
    \begin{scope}
        \clip (-12,0) rectangle (-6,6);
        \draw[fill=black!60] (-9,6) circle(2);
    \end{scope}
    \begin{scope}
        \clip (-12,6) rectangle (-6,10);
        \draw[fill=black!20] (-9,10) circle(1);
    \end{scope}
        \draw[densely dashed] (0,-12) -- (0,12);
        \draw[densely dashed] (6,-12) -- (6,12);
        \draw[densely dashed] (-6,-12) -- (-6,12);
        \draw[densely dashed] (-12,0) -- (12,0);
        \draw[densely dashed] (-12,6) -- (12,6);
        \draw[densely dashed] (-12,-6) -- (12,-6);
        \draw (-12,-12) -- (-12,12) -- (12,12) -- (12, -12) -- cycle;
    \end{tikzpicture}
    =
    \begin{tikzpicture}[scale=1/8 pt, thick, baseline={([yshift=-\the\dimexpr\fontdimen22\textfont2\relax] current bounding box.center)}]
      \begin{scope} 
        \clip (-12,9) rectangle (24+18,15); 
        \draw[fill=black!60] (-9,15) -- (-9,9) -- (3,9) -- (3,15) -- cycle;
        \draw[fill=black!60] (9,15) -- (9,9) -- (15,9) -- (15,15) -- cycle;
        \draw[fill=black!20] (2+184,15) -- (24+18,9) -- (15,9) -- (15,15) -- cycle;
        \draw[fill=black!20] (-3,15) circle(2);
        \end{scope}
        \begin{scope} 
        \clip (-12,1) rectangle (24+18,7);
        \draw[fill=black!60] (-9 + 6,7) ellipse (6 and 4.5);
        \draw[fill=black!60] (9,7) -- (9,1) -- (15,1) -- (15,7) -- cycle;
        \draw[fill=black!20] (24+18,7) -- (24+18,1) -- (15,1) -- (15,7) -- cycle;
        \end{scope}
        \begin{scope} 
        \clip (-12,-1) rectangle (24+18,-7);
        \draw[fill=black!60] (3+12,-7) -- (3+12,-1) -- (-3+12,-1) -- (-3+12,-7) -- cycle;
        \draw[fill=black!20] (3+12,-7) -- (3+12,-1) -- (24+18,-1) -- (24+18,-7) -- cycle;
        \draw[fill=black!60] (9+12+6,-7) ellipse (6 and 4.5);
        \end{scope}
        \begin{scope} 
        \clip (-12,-9) rectangle (24+18,-15);
        \draw[fill=black!60] (-9+18,-15) -- (-9+18,-9) -- (-3+18,-9) -- (-3+18,-15) -- cycle;
        \draw[fill=black!20] (3+18,-15) -- (3+18,-9) -- (-3+18,-9) -- (-3+18,-15) -- cycle;
        \draw[fill=black!60] (3+18,-15) -- (3+18,-9) -- (24+18,-9) -- (24+18,-15) -- cycle;
        \draw[fill=white] (9+18,-15) circle(2);
        \draw[fill=black!20] (24+18,-15) -- (24+18,-9) -- (15+18,-9) -- (15+18,-15) -- cycle;
        \draw[fill=black!60] (21+18,-15) circle(2);
        \end{scope}
      \draw[densely dashed] (-6,-1) -- (-6,-7);
      \draw[densely dashed] (-6,9) -- (-6,15);
      \draw[densely dashed] (-6,-9) -- (-6,-15);
      \draw[densely dashed] (0,-1) -- (0,-7); 
      \draw[densely dashed] (0,9) -- (0,15);
      \draw[densely dashed] (0,-9) -- (0,-15);
      \draw[densely dashed] (6,1) -- (6,7); 
      \draw[densely dashed] (6,-1) -- (6,-7);
      \draw[densely dashed] (6,9) -- (6,15);
      \draw[densely dashed] (6,-9) -- (6,-15);
      \draw[densely dashed] (12,9) -- (12,15);  
      \draw[densely dashed] (12,1) -- (12,7);
      \draw[densely dashed] (12,-1) -- (12,-7);
      \draw[densely dashed] (12,-9) -- (12,-15);
      \draw[densely dashed] (18,9) -- (18,15);  
      \draw[densely dashed] (18,1) -- (18,7);
      \draw[densely dashed] (18,-1) -- (18,-7);
      \draw[densely dashed] (18,-9) -- (18,-15);
      \draw[densely dashed] (18 + 6,9) -- (18+ 6,15);  
      \draw[densely dashed] (18+ 6,1) -- (18+ 6,7);
      \draw[densely dashed] (18+ 6,-9) -- (18+ 6,-15);
      \draw[densely dashed] (18 + 6+ 6,9) -- (18+ 6+ 6,15);  
      \draw[densely dashed] (18+ 6+ 6,1) -- (18+ 6+ 6,7);
      \draw[densely dashed] (18+ 6+ 6,-9) -- (18+ 6+ 6,-15);
      \draw[densely dashed] (18 + 6+ 6+ 6,9) -- (18+ 6+ 6+ 6,15);  
      \draw[densely dashed] (18+ 6+ 6+ 6,1) -- (18+ 6+ 6+ 6,7);
      \draw[densely dashed] (18+ 6+ 6+ 6,-1) -- (18+ 6+ 6+ 6,-7);
      \draw[densely dashed] (18+ 6+ 6+ 6,-9) -- (18+ 6+ 6+ 6,-15);
      \draw (-12,1) rectangle (24+18,7); 
      \draw (-12,9) rectangle (24+18,15);
      \draw (-12,-1) rectangle (24+18,-7);
      \draw (-12,-9) rectangle (24+18,-15);
    \end{tikzpicture}
\end{equation}


\medskip
We can thereafter define $\ff_S(\alpha) := \circ_{k = 1}^L\ff_S(_k\alpha) := \circ_{k +1}^L\otimes_{i_k = 1}^{T_k}\ff_S(_k\alpha_{i_k}),$ which completely determines $\ff$ on 2-morphisms, as $S$ is a $\Gamma$-fundamental solution. The data of a pseudofunctor requires for each 0-morphism $a \in \TLJ$, an invertible 2-morphism $\iota_a: \ff_S(id_a) \Rightarrow id_{\ff_S(a)}$ and for each pair of composable 1-morphisms $f,g \in \TLJ$, a natural invertible 2-morphism $\mu_{f,g}: \ff_S(f) \otimes \ff_S(g) \Rightarrow \ff_S(f \otimes g)$. By taking $\iota$ and $\mu$ to be identities, $\ff_S$ trivially satisfies the conditions required in order to be a pseudofunctor, which are namely coherence axioms between $\iota$ and $\mu$. Although we do not state these conditions here, we refer the interested reader to \cite{NLab2018Pseudofunctor}.
\end{proof}
The following proposition will allow us to simplify our classification problem, asserting it is sufficient to understand \textit{strict} $*$-pseudofunctors between (strict) $*$-2-categories, as opposed to the larger family of \textit{strong} $*$-pseudofunctors. Before we state the next result, we need to introduce the notion of equivalence of pseudofunctors between 2-categories. To do so, we must \textit{go up one level}, considering the 3-category \textbf{2Cat}. (See \cite{Betti1988Complete2-Categories} for more details on the construction of this category.)

\begin{definition}\label{pseudofunctor equivalence}
Consider pseudofunctors $\ff,\gg: \mathcal{B}\rightarrow \mathcal{C},$ between 2-categories $\mathcal B$ and $\mathcal C$. We say that $\ff$ and $\gg$ are equivalent if and only if there exist pseudonatural isomorphisms $\theta : \ff \Rightarrow \gg$ and $\kappa :\gg\Rightarrow \ff$ and invertible modifications in \textbf{2Cat}, $M:\theta\circ\kappa \Rrightarrow  \ONE_{\gg}$ and $M': \kappa\circ\theta \Rrightarrow \ONE_{\ff}.$ In the case where $\mathcal B$ and $\mathcal C$ are $*$-2-categories, and $\ff$ and $\gg$ are $*$-pseudofunctors, we also require the 2-cells in $M$ and $M'$ be unitary, as well all the 2-morphisms $\theta_e,\kappa_e,$ for every 1-morphism $e$ in $\mathcal C.$ See \cite{Leinster1998BasicBicategories} for a more detailed overview of modifications. We also refer the reader to the entries on pseudonatural transformations \cite{NLab2018Modifications} and modifications \cite{NLab2018PseudonaturalTransformation} on \textbf{nLab}.
\end{definition}
\begin{proposition}\label{strictify}
Every strong $*$-pseudofunctor $(\ff,\mu)$ from $\TLJ$ into a strict $*$-2-category $\mathcal C$ is unitarily equivalent to the canonical strict $*$-pseudofunctor $\ff_S$ generated by the fundamental solution $S = (J, H, C)$ defined as follows: 
\begin{align*}
    J^a &:= \ff(a) \text{, for every } a \in V(\Gamma),\\
    \hh^e &:= \ff(e) \text{ and}\\
    C^e &:= \mu^{-1}_{e,\overline{e}} \, \circ \, \ff(^e \cup ^{\overline{e}}) \text{, for every } e \in E(\Gamma).
\end{align*}
\end{proposition}

\begin{proof}

We shall construct the natural transformations $\theta$ and $\kappa$ from Definition \ref{pseudofunctor equivalence} together with the forementioned modifications between them and the corresponding identities. For an arbitrary object $a\in \rm{TLJ} (\Gamma)$, we define $\theta_a := \ONE_{\ff(a)},$ and $\kappa_a := \ONE_{\ff_S(a)} = \ONE_{\ff(a)}.$ Let us now consider an arbitrary 1-morphism $e = \otimes_{i=1}^N e_i$ in $\TLJ$. By the strictness hypothesis on $\mathcal C$, and strictness in $\TLJ,$ there is no loss of generality by choosing any preferred parenthesization when expanding the path $e$ or its image under $\ff_S$. In the following computation, we chose the rightmost grouping, obtaining:

\begin{align*}
    \ff\left(\bigotimes_{i = 1}^N e_i\right) &= \ff\left( e_1 \otimes \bigotimes_{i = 2}^N e_i\right)\\
    &\cong \mu_{e_1,\bigotimes_2^Ne_i} \left(\ff(e_1) \otimes \ff\left(\bigotimes_{i = 2}^N e_i\right)\right)\\
    &\cong \mu_{e_1,\bigotimes_2^Ne_i}\left(\ff(e_1)\otimes \mu_{e_2,\bigotimes_3^Ne_i} \left(\ff(e_2)\otimes \ff\left(\bigotimes_{i = 3}^N e_i\right)\right)\right) \\
    &\, \, \, \vdots\\
    &\cong \mu_{e_1,\bigotimes_2^Ne_i}\left(\ff_S(e_1)\otimes \mu_{e_2,\bigotimes_3^Ne_i}  \left(\ff_S(e_2)\otimes \mu_{e_3,\bigotimes_4^Ne_i} \hdots \mu_{e_{N-1}, e_N}\left(\ff_S(e_{N-1})\otimes \ff_S(e_N)\right)\right)\right).
\end{align*}
From this computation, we obtain the family of 2-morphism described as follows:
\begin{align*}
    \eta_e &:= \eta_{e_1,e_2,\hdots,e_N} : \ff_S(e)\Rightarrow \ff (e),\ \rm{defined}\ \rm{by} \\
    \eta_{e_1,e_2,\hdots,e_N} &:= \mu_{e_1, \otimes_2^N e_i}\circ (\ONE_{\ff(e_1)}\otimes \mu_{e_2,\otimes_3^N e_i})\circ (\ONE_{\ff(e_1)}\otimes \ONE_{\ff(e_2)}\otimes \mu_{e_3,\otimes_4^N e_i})\circ\hdots \\
    &\hdots\circ (\ONE_{\ff(e_1)}\otimes \ONE_{\ff(e_2)}\otimes\hdots\otimes \ONE_{\ff(e_{N-3}}\otimes \mu_{e_{N-2},e_{N-1}\otimes e_N})\circ\\
    &\circ(\ONE_{\ff(e_1)}\otimes \ONE_{\ff(e_2)}\otimes\hdots\otimes \ONE_{\ff(e_{N-2})}\otimes\mu_{e_{N-1},e_N}).
\end{align*}
By defining $(\eta_a :\ff_S(a)\rightarrow \ff(a)) := (\ONE_{\ff(a)})$ for $a\in V(\Gamma),$ we obtain a pseudonatural transformation $\eta\in 2 \, \textbf{Cat}(\ff_S\Rightarrow\ff).$ We shall sketch a proof of this assertion in short. In addition, since all the components of $\eta$ are invertible, this defines a pseudonatural equivalence from $\ff_S$ to $\ff.$ Notice that all the $\eta_e$ are manifestly unitary --as we are simply tensoring and composing the unitary 2-morphism $\mu$ with identities-- and that $\eta_{e_1,\hdots e_N}$ is natural in each of the $e_i$ conforming the path $e$. We now provide a complete outline for verifying that $\eta$ is a pseudonatural equivalence. It needs to be shown that $\eta$ is monoidal (with respect to the composition of 1-morphisms), respects units, and is natural (on 2-morphisms). We shall proceed in that order: \medskip\\
To prove $\eta$ is monoidal with respect to the one composition, we need to verify that for arbitrary composable 1-morphisms $e,f$ in $\mathcal {C}$, the following equality holds:

\begin{equation} \label{monoidalityidentity}
\begin{tikzcd}[every arrow/.append style={shift left}, row sep=large]
\ff_S(a) \arrow{d}[left]{\ONE_{\ff(a)}} \arrow{r}{\ff_S(e)} & \ff_S(b) \arrow[Rightarrow]{dl}{\eta_e} \arrow{d}{\ONE_{\ff(b)}} \arrow{r}{\ff_S(f)} & \ff_S(c) \arrow[Rightarrow]{dl}{\eta_f} \arrow{d}{\ONE_{\ff(c)}}\\
\ff(a) \arrow{r}[below]{\ff(e)} \arrow[bend right=45, ""{name=P, below}]{rr}[below]{\ff(e \otimes f)} & \ff(b) \arrow[Rightarrow, to=P,"\mu^\ff_{e,f}", xshift=-0.6ex, yshift=0.2ex] \arrow{r}[below]{\ff(f)} & \ff(c)
\end{tikzcd}
=
\begin{tikzcd}[every arrow/.append style={shift left}]
&\ff_S(b) \arrow{dr}{\ff_S(f)} & \\
\ff_S(a) \arrow{d}[left]{\ONE_{\ff(a)}} \arrow{ur}{\ff_S(e)} \arrow{rr}[below, ""{name=U, above}]{\ff_S(e \otimes f)} \arrow[Leftarrow, from=U]{ur}{\ONE} & &\ff_S(c) \arrow[Rightarrow, yshift=0.7ex]{dll}{\eta_{e \otimes f}} \arrow{d}{\ONE_{\ff(c)}}\\
\ff(a) \arrow{rr}[below]{\ff(e \otimes f)} & & \ff(c) 
\end{tikzcd}
\end{equation}
However, this follows immediately from the graphical calculation: 

\begin{equation}
    \begin{tikzpicture}[scale = 2/5pt, baseline={([yshift=-\the\dimexpr\fontdimen22\textfont2\relax] current bounding box.center)}]
    \draw (-1,-5+7.2) -- (-1,-6.2+7.2);
    \draw (0 + 2.75,-5 +7.2) -- (0 + 2.75,-6.2+7.2);
    \draw (-0.3 + 2.75,-6.2+7.2) rectangle (0.3 + 2.75,-5.0+7.2);
    \draw (-1.3,-6.2+7.2) rectangle (-0.7,-5.0+7.2);
    \node at (0.85,-5.6+7.2) {\footnotesize$\Big) \otimes \ff_S\Big($};
    \node at (-2.2,-5.6+7.2) {\footnotesize$\ff_S\Big($};
    \node at (0.7 + 2.75 + 2,-5.6+7.2) {\footnotesize$\Big) \otimes \hdots \otimes \ff_S\Big($};
    \draw (-1.3 + 9.2,-6.2+7.2) rectangle (-0.7 + 9.2,-5.0+7.2);
    \draw (-1 + 9.2, -6.2+7.2) -- (-1 + 9.2, -5.0 + 7.2);
    \node at (-0.3 + 9.2,-5.6+7.2) {\footnotesize$\Big)$};
    \draw (2.75, -6.2+7.2 + 8) -- (2.75, -6.2+7.2 + 7);
    \draw (2.75, -6.2+7.2 + 7) -- (-1, -6.2+7.2 + 6);
    \draw (-1, -6.2+7.2 + 6) -- (-1, -5+7.2);
    \draw (2.75, -6.2+7.2 + 7) -- (2.75+3.75, -6.2+7.2 + 6);
    \draw (2.75+3.75, -6.2+7.2 + 6) --  (-1+3.75, -6.2+7.2 + 5);
    \draw (-1+3.75, -6.2+7.2 + 5) -- (-1+3.75, -5+7.2);
    \draw (-1 + 9.2, -6.2+7.2 + 5 -0.5) -- (-1 + 9.2, -5+7.2);
    \node at ( 7.35, 6.5) {\footnotesize$\ddots$};
    \draw (2.75+3.75, -6.2+7.2 + 6) -- (2.75+3.75 + 0.3, -6.2+7.2 + 6 - 0.2);
    \draw (-1 + 9.2, -6.2+7.2 + 5 -0.5) -- (-1 + 9.2 - 0.3, -6.2+7.2 + 5 -0.5 + 0.3);
    \end{tikzpicture}
    = \,\,\,\,\,\hdots\,\,\,\,\, =
    \begin{tikzpicture}[scale = 2/5pt, baseline={([yshift=-\the\dimexpr\fontdimen22\textfont2\relax] current bounding box.center)}]
    \draw (0 + 1 + 2.75,1) -- (0 + 1 + 2.75,0);
    \draw (0 + 1 + 2.75,0) -- (-1 + 1 + 2.75,-1);
    \draw (-1 + 1 + 2.75,-1) -- (-1 + 1 + 2.75,-2);
    \draw (0 + 1 + 2.75,0) -- (1 + 1 + 2.75,-1);
    \draw (0 + 1 + 2.75,-1) -- (0 + 1 + 2.75,-2);
    \draw (0 + 1 + 2.75,-1) -- (0.5 + 1 + 2.75,-0.5);
    \draw (3 + 1 + 2.75 - 1,-3 + 1) -- (4 + 1 + 2.75 - 1,-4 + 1);
    \draw (3 + 1 + 2.75 - 1,-4 + 1) -- (3 + 1 + 2.75 - 1,-5 + 1);
    \draw (3 + 1 + 2.75 - 1,-4 + 1) -- (3.5 + 1 + 2.75 - 1,-3.5 + 1);
    \draw (4 + 1 + 2.75 - 1,-4 + 1) -- (4 + 1 + 2.75 - 1,-5 + 1);
    \draw (-1 + 1 + 2.75,-3) -- (-1 + 1 + 2.75,-4);
    \draw (0 + 1 + 2.75,-3) -- (0 + 1 + 2.75,-4);
    \node at (-1 + 1 + 2.75,-2.3) {\footnotesize$\vdots$};
    \node at (0 + 1 + 2.75,-2.3) {\footnotesize$\vdots$};
    \node at (1.9 + 1 + 2.75 - 0.4,-1.75 + 0.5) {\footnotesize$\ddots$};
    \node at (1.9 + 1 + 2.75 - 0.4 -3.5 - 0.7,-1.75 + 0.5-1) {\footnotesize$\eta_{e_2,\hdots,e_N}$};
    \node at (1 + 1 + 2.75,-3.5) {\footnotesize$\hdots$};
    \draw (-1,-5+7.2) -- (-1,-6.2+7.2);
    \draw (0 + 2.75,-5 +7.2) -- (0 + 2.75,-6.2+7.2);
    \draw (-0.3 + 2.75,-6.2+7.2) rectangle (2.4 + 2.75,-5.0+7.2);
    \draw (-1.3,-6.2+7.2) rectangle (-0.7,-5.0+7.2);
    \node at (0.85,-5.6+7.2) {\footnotesize$\footnotesize \Big) \otimes \ff_S\Big($};
    \draw (2 + 2.75,-6.2+7.2) -- (2 + 2.75,-5.0+7.2);
    \node at (1 + 2.75,-5.6+7.2) {\footnotesize$\hdots$};
    \node at (-2.2,-5.6+7.2) {\footnotesize$\ff_S\Big($};
    \node at (2.8 + 2.75,-5.6+7.2) {\footnotesize$\Big)$};
    \draw (-1,1) -- (-1,-2);
    \node at (-1,-2.3) {\footnotesize$\vdots$};
    \draw (-1,-3) -- (-1,-4);
    \draw (-1,-5+7.2) -- (2.75/2,-5+7.2+1);
    \draw (2.75/2,-5+7.2+1) -- (1 + 2.75,-5+7.2);
    \draw (2.75/2,-5+7.2+1) -- (2.75/2,-5+7.2+1+1);
    \end{tikzpicture}
    =
    \begin{tikzpicture}[scale = 2/5pt, baseline={([yshift=-\the\dimexpr\fontdimen22\textfont2\relax] current bounding box.center)}]
    \draw (0 + 1,2.2) -- (0 + 1,3.2);
    \draw (0 + 1,1) -- (0 + 1,0);
    \draw (0 + 1,0) -- (-1 + 1,-1);
    \draw (-1 + 1,-1) -- (-1 + 1,-2);
    \draw (0 + 1,0) -- (1 + 1,-1);
    \draw (0 + 1,-1) -- (0 + 1,-2);
    \draw (0 + 1,-1) -- (0.5 + 1,-0.5);
    \draw (1 + 1,-1) -- (1.5 + 1,-1.5);
    \draw (3 + 1,-3) -- (2 + 1,-4);
    \draw (3 + 1,-3) -- (2.5 + 1,-2.5);
    \draw (3 + 1,-3) -- (4 + 1,-4);
    \draw (3 + 1,-4) -- (3 + 1,-5);
    \draw (3 + 1,-4) -- (3.5 + 1,-3.5);
    \draw (4 + 1,-4) -- (4 + 1,-5);
    \draw (-1 + 1,-3) -- (-1 + 1,-5);
    \draw (0 + 1,-3) -- (0 + 1,-5);
    \draw (2 + 1,-4) -- (2 + 1,-5);
    \node at (-1 + 1,-2.3) {\footnotesize$\vdots$};
    \node at (0 + 1,-2.3) {\footnotesize$\vdots$};
    \node at (1.9 + 1,-1.75) {\footnotesize$\ddots$};
    \node at (1 + 1,-4.5) {\footnotesize$\hdots$};
    \draw (0,-5 +7.2) -- (0,-6.2+7.2);
    \draw (-1,-5+7.2) -- (-1,-6.2+7.2);
    \draw (-1.3,-6.2+7.2) rectangle (4.4,-5.0+7.2);
    \draw (2,-6.2+7.2) -- (2,-5.0+7.2);
    \draw (3,-6.2+7.2) -- (3,-5.0+7.2);
    \draw (4,-6.2+7.2) -- (4,-5.0+7.2);
    \node at (1,-5.6+7.2) {\footnotesize$\hdots$};
    \node at (-2.2,-5.6+7.2) {\footnotesize$\ff_S\Big($};
    \node at (4.8,-5.6+7.2) {\footnotesize$\Big)$};
    \node at (-2.2 + 1 - 0.7,-2.3 - 1) {\footnotesize$\eta_{e_1,\hdots,e_N}$};
    \end{tikzpicture}
\end{equation}

\noindent That $\eta$ respects units follows easily, as if $g\in E(\Gamma),$ then $\eta_g = \ONE_{\ff(g)},$ is the tensor identity.

Finally, to see that $\eta$ is natural, let $g = \otimes_1^N g_i $ and $f = \otimes_1^M f_j $ be 1-morphisms in TLJ$(\Gamma)$ and $\alpha: f\Rightarrow g \in 2$TLJ$(\Gamma)(f\Rightarrow g)$ be an arbitrary $\TLJ$ diagram from $f$ to $g$. We shall verify the following identity holds:
\begin{equation} \label{naturalityidentity}
  \begin{tikzcd}[every arrow/.append style={shift left}, column sep=large, row sep=large]
    \ff_S(a) \arrow{r}[below, ""{name=Lower,inner sep=1pt}]{\ff_S(g)} \arrow[bend left=45,""{name=Upper,inner sep=1pt,below}]{r}{\ff_S(f)} \arrow[Rightarrow, from=Upper, to=Lower, "\ff_S(\alpha)"{left}, xshift=-0.2ex] \arrow{d}[left]{\ONE_{\ff(a)}} & \ff_S(b) \arrow[Rightarrow]{dl}[above,sloped]{\eta_g} \arrow{d}[right]{\ONE_{\ff(b)}} \\
   \ff(a) \arrow{r}[below]{\ff(g)} & \ff(b)
  \end{tikzcd}
=
  \begin{tikzcd}[every arrow/.append style={shift left}, column sep=large, row sep=large]
    \ff_S(a) \arrow{r}{\ff_S(f)} \arrow{d}[left]{\ONE_{\ff(a)}} & \ff_S(b) \arrow[Rightarrow,sloped]{dl}[above]{\eta_f} \arrow{d}[right]{\ONE_{\ff(b)}} \\
   \ff(a) \arrow{r}[above, ""{name=Upper,inner sep=1pt,below}]{\ff(f)} \arrow[bend right=45,""{name=Lower,inner sep=1pt}]{r}[below]{\ff(g)} \arrow[Rightarrow, from=Upper, to=Lower, "\ff(\alpha)"{left}, xshift=-0.2ex] & \ff(b) 
  \end{tikzcd}
\end{equation}
In order to do so, we decompose $\alpha$ as in the proof of Proposition \ref{canonical} obtaining $\alpha = \circ_{k=1}^L\otimes_{i_k = 1}^{T_k} {}_k\alpha_{i_k}.$ Recall that each ${}_k\alpha_{i_k}$ is either empty, a single string, or a single cup/cap. Since $\ff_S$ is a strict functor, proving the naturality of $\eta$ then reduces to show that equality \ref{naturalityidentity} holds for each of the $_k\alpha_{i_k}.$ For the (colored) empty or the single string diagram cases, equality holds trivially. In the cup/cap cases, equality follows from the naturality of the given ``tensorator" data from  $(\ff,\mu^\ff)$. 




We now  define $\theta$ and $\kappa$ simply as
$$\theta := \eta\ \rm{and}\ \kappa := \eta^{-1},$$
so automatically they are both unitary pseudonatural equivalences $\theta: \ff_S\Rightarrow \ff$ and $\kappa: \ff\Rightarrow \ff_S$. We are now ready to provide the data for the modifications $M:\kappa\circ\theta \Rrightarrow \ONE_{\ff_S}$ and $M': \theta\circ\kappa \Rrightarrow \ONE_{\ff}.$ For each object $a\in \rm{TLJ}(\Gamma),$ we observe that $(\kappa\circ\theta)_a = \ONE_{\ff(a)}\circ \ONE_{\ff(a)} = \ONE_{\ff(a)}$ and that $\ONE_{\ff}(a) = \ONE_{\ff(a)},$ so we define 
\begin{align*}
M_a &: (\kappa\circ\theta)_a \Rightarrow \ONE_{\ff(a)},\ by\\
M_a &:= \ONE_{\ONE_{\ff(a)}}.    
\end{align*}
Similarly we define
\begin{align*}
    M'_a &: (\theta\circ\kappa)_a \Rightarrow \ONE_{\ff_S(a)},\ by \\
    M'_a &:= \ONE_{\ONE_{\ff_S(a)}} = \ONE_{\ONE_{\ff(a)}}.
\end{align*}

It is routine to verify this data defines a modification. (We warn the reader we explicitly omitted all left and right unitors.) By observing these are all isomorphisms, we conclude $M$ and $M'$ describe the desired equivalence, thus completing the proof.
\end{proof}

\begin{proposition} \label{funda1}
 Let $(V,E,C)$ and $(W,I,D)$ be two $\Gamma$-fundamental solutions in $\ucat$, and let $\ff$ and $\gg$ be the unique unitary modules determined by each, respectively. Moreover, lets assume we have the following data:
 \begin{itemize}
     \item  for every $a\in V(\Gamma)$ we have that $\ff(a) = V^a = W^a = \gg(a),$ together with trivial 1-morphisms in $\ucat:$ $\theta_a = \kappa_a = \ONE_{V^a} ,$ and $M_a = M'_a = \ONE_{\ONE_{V^a}};$
     \item for every edge $e\in$ 1TLJ($\Gamma$)$(a\rightarrow b)$ two families of unitary 2-morphisms in $\ucat$: $\theta_e: I^e\Rightarrow E^e$ and $\kappa_e: E^e\Rightarrow I^e$ such that and $(\kappa_e)^* = \theta_e$, together with the conditions $\theta_{\ONE_{V^a}} = \kappa_{\ONE_{V^a}} = \ONE_{\ONE_{V^a}}$. (We remark that  $\ONE_{\ONE_{V^a}}$ also acts as the tensor unit.)
 \end{itemize}
 
 Then $\ff$ and $\gg$ are equivalent via the pseudonatural isomorphisms $\theta: \ff\Rightarrow \gg$ and $\kappa:\gg\Rightarrow \ff$ and the invertible modification $M':(\kappa\circ\theta)\Rrightarrow \ONE_{\ff}$
 if and only if for each 
 edge $e\in$ 1TLJ($\Gamma$)$(a\rightarrow b)$ we have the unitary 2-morphisms in $\ucat$ $\theta_e$, $\kappa_e$  satisfying $(\theta_e \otimes \theta_{\bar{e}})\circ D^e = C^ e$.
This is represented diagrammatically as follows:\medskip
\begin{equation}
\begin{tikzpicture}[scale=0.75, every node/.style={scale=0.75}, semithick,baseline={([yshift=-\the\dimexpr\fontdimen22\textfont2\relax] current bounding box.center)}]
  \draw (-1,0) .. controls (-1,-0.555) and (-0.555,-1) .. (0,-1)
               .. controls (0.555,-1) and (1,-0.555) .. (1,0);
  \draw (-1,0) -- (-1,0.8);
  \draw (1,0) -- (1,0.8);
  \draw (0.1,-1.2) -- (-0.1,-1) -- (0.1, -0.8);
  \node at (0,-1.4) {\large$C^e$};
  \node at (-1,1.1) {\large$E^e$};
  \node at (1,1.1) {\large$E^{\overline{e}}$};
\end{tikzpicture}
= \,\,\,
\begin{tikzpicture}[scale=0.75, every node/.style={scale=0.75}, semithick,baseline={([yshift=-\the\dimexpr\fontdimen22\textfont2\relax] current bounding box.center)}]
  \draw (-1,0) .. controls (-1,-0.555) and (-0.555,-1) .. (0,-1)
               .. controls (0.555,-1) and (1,-0.555) .. (1,0);
  \draw (-1.35,0) rectangle (-0.65,0.6);
  \draw (1.45,0) rectangle (0.65,0.6);
  \draw (-1,0.6) -- (-1,0.8);
  \draw (1,0.6) -- (1,0.8);
  \draw (0.1,-1.2) -- (-0.1,-1) -- (0.1, -0.8);
  \node at (0,-1.4) {\large$D^e$};
  \node at (-1, 0.3) {\large$\theta_e$};
  \node at (1, 0.3) {\large$\theta_{\overline{e}}$};
  \node at (-1,1.1) {\large$E^e$};
  \node at (1,1.1) {\large$E^{\overline{e}}$};
\end{tikzpicture}
\end{equation}
\end{proposition}

\begin{proof}
 We begin by proving the forward direction. Consider pseudonatural transformations $\theta$ and $\kappa$ and the modification $M'$ as in the statement. For each $e\in$ 1TLJ($\Gamma$)$(a\rightarrow b),$ we define a two-morphism in $\ucat$ by $H^e := (M'_b)\circ(\theta_e\otimes \kappa_{\bar{e}}^{*}),$ which is manifestly unitary. We then obtain the following chain of equations which are heavily guided by the graphical calculus in $\ucat$:
 \begin{align*}
     H^e\circ (\ONE_{\theta_a}\otimes D^e\otimes \ONE_{\kappa_a})\circ (H^{\ONE_a})^* 
     &= H^e\circ(\ONE_{\theta_a}\otimes \kappa_{e\otimes \bar{e}})\circ(\ONE_{\theta_a}\otimes \ONE_{\kappa_a}\otimes C^{e})\circ({\theta_{\ONE_a}}^*\otimes \ONE_{\kappa_a})\circ (M'_a)^* \\
     &= (\ONE_{E^e} \otimes M'_b \otimes \ONE_{E^{\bar{e}}})\circ(\theta_e\otimes \ONE_{\kappa_b}\otimes \ONE_{E^{\bar{e}}}) \circ(\ONE_{\theta_a}\otimes\kappa_e\otimes \ONE_{E^{\bar{e}}})\circ\\
     &\ \ \ \ \ \circ(\ONE_{\theta_a}\otimes \ONE_{\kappa_a} \otimes C^e)\circ({\theta_{\ONE_a}}^*\otimes \ONE_{\kappa_a})\circ (M'_a)^*\\
     &= (M'_a\otimes \ONE_{E^e}\otimes \ONE_{E^{\bar{e}}})\circ (\ONE_{\theta_a}\otimes \ONE_{\kappa_a}\otimes C^e)\circ({\theta_{\ONE_a}}^*\otimes \ONE_{\kappa_a})\circ (M'_a)^* \\
     &= C^e,
 \end{align*}
 showing that $C^e$ and $D^e$ are unitary conjugates. Now, since we are assuming that $\theta_a = \ONE_{V^a} = \kappa_a$ and $M'_a = \ONE_{\ONE_{V^a}},$ the unitary $H^e$ can explicitly be expressed as $H^e = \theta_e\otimes {\kappa_{\bar{e}}}^*,$ giving the desired relation.
\begin{equation}
     \begin{tikzpicture} [scale=1/3 pt, thick, baseline={([yshift=-\the\dimexpr\fontdimen22\textfont2\relax] current bounding box.center)}]
       \tikzstyle{block} = [rectangle, draw=black, very thick, text centered,rounded corners, minimum height=2em];
       \matrix [matrix of math nodes,
        column sep={1.5cm,between origins},
        row sep={1.25cm,between origins},
        nodes={block, draw, minimum size=6.5mm}]
        {
                                      & |(Mb)| {M}_b'             & \\
        |(Te)| \theta_e               &                        & |(kei)| \kappa_{\overline{e}}^{-1} \\
                                      & |(De)| D^e             & \\
        |(Tida)| \theta_{\ONE_a}^{-1}   &                        & |(Kida)| \kappa_{\ONE_a}\\
                                      & |(Ma)| {M'}_{a}^{-1}   & \\
        };
       \draw (De) to (Te)
             (De) to (kei)
             (Te) to (Tida)
             (kei) to (Kida)
             (Tida) to (Ma)
             (Kida) to (Ma)
             (Te) to (Mb)
             (kei) to (Mb);
        \node (triv)    [below of=De]  {};
        \node (trash1)  [above of=Te]  {};
        \node (trash2)  [above of=kei] {};
        \node (2trash1) [above of=trash1]  {};
        \node (2trash2) [above of=trash2] {};
        \node (3trash1) [above = 0.5cm of 2trash1] {};
        \node (3trash2) [above = 0.5cm of 2trash2] {};
        \node (up1) [above = 0.25cm of 3trash1]  {};
        \node (up2) [above = 0.25cm of 3trash2]  {};
        \node (triv3) [below = 0.5cm of Ma] {};
        \node (dtrash1) [below of=Tida] {};
        \node (dtrash2) [below of=Kida] {};
        \node (2dtrash1) [below of=dtrash1] {};
        \node (2dtrash2) [below of=dtrash2] {};
        \node (3dtrash1) [below = 0.5cm of 2dtrash1] {};
        \node (3dtrash2) [below = 0.5cm of 2dtrash2] {};
        \node (down1) [below = 0.25cm of 3dtrash1] {};
        \node (down2) [below = 0.25cm of 3dtrash2] {};
        \node (triv2) [above = 0.5cm of Mb] {};
        \draw[dashed] (Tida) to (triv)
                      (Kida) to (triv)
                      (triv) to (De);
        \draw[dashed] (Mb) to (triv2)
                      (triv2) to (3trash1)
                      (triv2) to (3trash2);
        \draw (Te) to (up1)
              (kei) to (up2);
        \draw[dashed] (Ma) to (triv3)
                      (triv3) to (3dtrash1)
                      (triv3) to (3dtrash2);
        \draw[dashed] (Tida) to (down1)
              (Kida) to (down2);
        \draw[densely dashed,blue,rounded corners] (-2,-6) rectangle (7,6);
     \end{tikzpicture}
     \,\,\,=
     \begin{tikzpicture} [scale=1/3 pt, thick, baseline={([yshift=-\the\dimexpr\fontdimen22\textfont2\relax] current bounding box.center)}]
       \tikzstyle{block} = [rectangle, draw=black, very thick, text centered,rounded corners, minimum height=2em];
       \matrix [matrix of math nodes,
        column sep={1.5cm,between origins},
        row sep={1.25cm,between origins},
        nodes={block, draw, minimum size=6.5mm}]
        {
                                      & |(Mb)| {M}_b'             & \\
        |(Te)| \theta_e               &                        & |(kei)| \kappa_{\overline{e}}^{-1} \\
                                      & |(De)| \kappa_{e \otimes \overline{e}}             & \\
        |(Tida)| \theta_{\ONE_a}^{-1}   &                        & |(Kida)| C^e\\
                                      & |(Ma)| {M'}_{a}^{-1}   & \\
        };
       \draw (De) to (Te)
             (Te) to (Tida)
             (Tida) to (Ma)
             (Te) to (Mb)
             (kei) to (Mb);
        \draw[double distance=2pt] (De) to (kei);
        \node (triv)    [below of=De]  {};
        \node (trash1)  [above of=Te]  {};
        \node (trash2)  [above of=kei] {};
        \node (2trash1) [above of=trash1]  {};
        \node (2trash2) [above of=trash2] {};
        \node (3trash1) [above = 0.5cm of 2trash1] {};
        \node (3trash2) [above = 0.5cm of 2trash2] {};
        \node (up1) [above = 0.25cm of 3trash1]  {};
        \node (up2) [above = 0.25cm of 3trash2]  {};
        \node (triv3) [below = 0.5cm of Ma] {};
        \node (dtrash1) [below of=Tida] {};
        \node (dtrash2) [below of=Kida] {};
        \node (2dtrash1) [below of=dtrash1] {};
        \node (2dtrash2) [below of=dtrash2] {};
        \node (3dtrash1) [below = 0.5cm of 2dtrash1] {};
        \node (3dtrash2) [below = 0.5cm of 2dtrash2] {};
        \node (down1) [below = 0.25cm of 3dtrash1] {};
        \node (down2) [below = 0.25cm of 3dtrash2] {};
        \node (triv2) [above = 0.5cm of Mb] {};
        \draw[dashed] (Tida) to (triv);
        \draw[dashed] (Mb) to (triv2)
                      (triv2) to (3trash1)
                      (triv2) to (3trash2);
        \draw (Te) to (up1)
              (kei) to (up2);
        \draw[dashed] (Ma) to (triv3)
                      (triv3) to (3dtrash1)
                      (triv3) to (3dtrash2);
        \draw [dashed] (Tida) to (down1)
                       (Kida) to (down2);
        \draw (De) to (Ma);
        \draw [double distance=2pt] (De) to (Kida);
        \draw[densely dashed,blue,rounded corners] (-2,-2) rectangle (7,6);
     \end{tikzpicture}
     \,\,\,=\,\,\,
     \begin{tikzpicture} [scale=1/3 pt, thick, baseline={([yshift=-\the\dimexpr\fontdimen22\textfont2\relax] current bounding box.center)}]
       \tikzstyle{block} = [rectangle, draw=black, very thick, text centered,rounded corners, minimum height=2em];
       \matrix [matrix of math nodes,
        column sep={1.5cm,between origins},
        row sep={1.25cm,between origins},
        nodes={block, draw, minimum size=6.5mm}]
        {
                                      & |(Mb)| {M}_b'             & \\
        |(Te)| \theta_e               &                        & \\
                                      & |(De)| \kappa_{e}             & \\
        |(Tida)| \theta_{\ONE_a}^{-1}   &                        & |(Kida)| C^e\\
                                      & |(Ma)| {M'}_{a}^{-1}   & \\
        };
       \draw (De) to (Te)
             (Te) to (Tida)
             (Tida) to (Ma)
             (Te) to (Mb)
             (De) to (Mb);
        \node (megatrash) [above of=Kida] {};
        \node (megatrash2) [above of=megatrash] {};
        \node (kei) [above of=megatrash2] {};
        \node (triv)    [below of=De]  {};
        \node (trash1)  [above of=Te]  {};
        \node (trash2)  [above of=kei] {};
        \node (2trash1) [above of=trash1]  {};
        \node (2trash2) [above = 0.50cm of trash2] {};
        \node (3trash1) [above = 0.5cm of 2trash1] {};
        \node (3trash2) [above = 0.25cm of 2trash2] {};
        \node (up1) [above = 0.25cm of 3trash1]  {};
        \node (up2) [above = 0.25cm of 3trash2]  {};
        \node (triv3) [below = 0.5cm of Ma] {};
        \node (dtrash1) [below of=Tida] {};
        \node (dtrash2) [below of=Kida] {};
        \node (2dtrash1) [below of=dtrash1] {};
        \node (2dtrash2) [below of=dtrash2] {};
        \node (3dtrash1) [below = 0.5cm of 2dtrash1] {};
        \node (3dtrash2) [below = 0.5cm of 2dtrash2] {};
        \node (down1) [below = 0.25cm of 3dtrash1] {};
        \node (down2) [below = 0.25cm of 3dtrash2] {};
        \node (triv2) [above = 0.5cm of Mb] {};
        \draw[dashed] (Tida) to (triv);
        \draw[dashed] (Mb) to (triv2)
                      (triv2) to (3trash1)
                      (triv2) to (3trash2);
        \draw (Te) to (up1)
              (Kida) to (up2);
        \draw[dashed] (Ma) to (triv3)
                      (triv3) to (3dtrash1)
                      (triv3) to (3dtrash2);
        \draw [dashed] (Tida) to (down1)
                       (Kida) to (down2);
        \draw (De) to (Ma);
        \draw (De) to (Kida);
        \draw[densely dashed,blue,rounded corners] (-6,-2) rectangle (2,10);
     \end{tikzpicture}
     \end{equation}
     \begin{equation*}
     = \,\,\,\,\,
     \begin{tikzpicture} [scale=1/3 pt, thick, baseline={([yshift=-\the\dimexpr\fontdimen22\textfont2\relax] current bounding box.center)}]
       \tikzstyle{block} = [rectangle, draw=black, very thick, text centered,rounded corners, minimum height=2em];
       \matrix [matrix of math nodes,
        column sep={1.5cm,between origins},
        row sep={1.25cm,between origins},
        nodes={block, draw, minimum size=6.5mm}]
        {
                                      &                        & \\
                                      &                        & \\
                                      & |(De)| M_a'     & \\
        |(Tida)| \theta_{\ONE_a}^{-1}   &                        & |(Kida)| C^e\\
                                      & |(Ma)| {M'}_{a}^{-1}   & \\
        };
       \draw (Tida) to (Ma);
        \node (megatrash) [above of=Kida] {};
        \node (megatrash2) [above = 0.5cm of megatrash] {};
        \node (kei) [above = 0.25cm of megatrash2] {};
        \node (triv)    [below of=De]  {};
        \node (trash0)  [above of=Tida]  {};
        \node (trash1)  [above of=trash0] {};
        \node (trash2)  [above = 0.5cm of kei] {}; 
        \node (triv3) [below = 0.5cm of Ma] {};
        \node (dtrash1) [below of=Tida] {};
        \node (dtrash2) [below of=Kida] {};
        \node (2dtrash1) [below of=dtrash1] {};
        \node (2dtrash2) [below of=dtrash2] {};
        \node (3dtrash1) [below = 0.25cm of 2dtrash1] {};
        \node (3dtrash2) [below = 0.25cm of 2dtrash2] {};
        \node (down1) [below = 0.5cm of 3dtrash1] {};
        \node (down2) [below = 0.5cm of 3dtrash2] {};
        \draw[dashed] (Tida) to (triv);
                      (triv2) to (3trash1)
                      (triv2) to (3trash2);
              (Kida) to (up2);
        \draw[dashed] (Ma) to (triv3)
                      (triv3) to (3dtrash1)
                      (triv3) to (3dtrash2);
        \draw [dashed] (Tida) to (down1)
                       (Kida) to (down2);
        \draw (De) to (Ma);
        \draw (Tida) to (De);
        \draw[double distance=4pt] (Kida) to (trash2);
        \draw[dashed] (De) to (kei);
        \draw[densely dashed,blue,rounded corners] (-6,-9) rectangle (3,3);
     \end{tikzpicture}
     \,\,\,\,\, =
     \begin{tikzpicture} [scale=1/3 pt, thick, baseline={([yshift=-\the\dimexpr\fontdimen22\textfont2\relax] current bounding box.center)}]
       \tikzstyle{block} = [rectangle, draw=black, very thick, text centered,rounded corners, minimum height=2em];
       \matrix [matrix of math nodes,
        column sep={1.5cm,between origins},
        row sep={1.25cm,between origins},
        nodes={block, draw, minimum size=6.5mm}]
        {
                                      &                        & \\
                                      &                        & \\
                                      &                        & \\
                                      &                        & |(Kida)| C^e\\
                                      &                        & \\
        };
        \node (megatrash) [above of=Kida] {};
        \node (megatrash2) [above of=megatrash] {};
        \node (kei) [above = 0.25cm of megatrash2] {};
        \node (trash2)  [above = 0.25cm of kei] {}; 
        \node (triv3) [below of=Ma] {};
        \node (dtrash2) [below of=Kida] {};
        \node (dtrash1) [left of=Ma] {};
        \node (2dtrash2) [below of=dtrash1] {};
        \node (2dtrash2) [below of=dtrash2] {};
        \node (3dtrash1) [below = 0.5cm of 2dtrash1] {};
        \node (3dtrash2) [below = 0.5cm of 2dtrash2] {};
        \node (down1) [below = 0.25cm of 3dtrash1] {};
        \node (down2) [below = 0.25cm of 3dtrash2] {};
        \draw [dashed] (3dtrash1) to (down1)
                       (3dtrash1) to (2dtrash2)
                       (Kida) to (down2);
        \draw[double distance=4pt] (Kida) to (trash2);
     \end{tikzpicture}
\end{equation*}

 Conversely, we shall construct the necessary pseudonatural transformations and modifications from the given data. First, if $e=\otimes e_i$ is a path in $\Gamma,$ we can regard $e$ as a 1-morphism in TLJ$(\Gamma),$ as a reduced word; i.e. containing no identities to suppress via the left or right unitors. We can then define $\theta_e := \otimes \theta_{e_i},$ extending to every 1-morphism in TLJ$(\Gamma)$. We proceed similarly with $\kappa,$ obtaining a unitary 2-morphism in $\ucat$ for every 1-morphism $e$ in TLJ$(\Gamma).$ We shall now verify that $\theta$ is a pseudonatural isomorphism. Since by definition it respects units and is monoidal with respect to 1-composition, it remains to see it is natural in 2-morphisms. However, to see this it suffices to check naturality for single cups as single strands, as we did in the proof of the previous proposition. Thus, that $\theta$ and $\kappa$ are natural follows from simple computation, using that $(\theta_e \otimes \theta_{\bar{e}})\circ D^e = C^ e$.
 
 Finally, that $M'$ defines a modification $M': (\kappa\circ\theta) \Rrightarrow \ONE_{\ff},$ follows directly from the hypothesis $(\kappa_e)^* = \theta_e$ for every edge in $\Gamma,$ as it directly translates to the commuting two-cell in the definition of a modification. This completes the proof.
\end{proof} 

Since $\ucat$ is $*$-2-equivalent to $\hilb$, classifying $\Gamma$-fundamental solutions in $\hilb$ --under the hypotheses of the previous proposition-- is equivalent to classifying $*$-pseudofunctors $\TLJ\rightarrow \ucat$. We then get the following corollary from the equivalence of categories described in Proposition \ref{catequivalence} and from Proposition \ref{funda1}:
\begin{corollary} \label{funda2} 

Consider two $\Gamma$-fundamental solutions $S =(J,H,C)$ and $T = (J, \tilde{H},\tilde{C})$ in $\hilb$ such that for each $a\in V(\Gamma),$ we have $\hh^{\ONE_a} = \tilde{\hh}^{\ONE_a},$ and $C^{\ONE_a} = \tilde{C}^{\ONE_a}.$ Furthermore, pushing forward these solutions using the equivalence $\Theta$ introduced in Proposition \ref{catequivalence}, we obtain $\Gamma$-fundamental solutions $\Theta[S] := (\Theta[J], \Theta[\hh], \Theta[C])$ and $\Theta[T] := (\Theta[J], \Theta [\tilde{\hh}], \Theta[\tilde{C}])$ in $\ucat$, defining corresponding strict $*$-pseudofunctors in $ \ucat,$ denoted $\ff$ and $\gg$. 

Then $\ff$ and $\gg$ are unitarily equivalent modules if and only if  for each edge $e\in E(\Gamma)$, there exists unitary  isomorphisms  $U^e\in 2\hilb(\tilde{\hh^e}\Rightarrow \hh^e)$ such that
$$C^e = ( U^e \otimes U^{\overline{e}} ) \circ \tilde{C}^e.$$
\end{corollary}

\begin{proof}
It is easy to see that $\Theta[S]$ and $\Theta[T]$ define $\Gamma$-fundamental solutions in $\ucat$, and this follows from the  monoidality of $\Theta.$ Hence, by Proposition (\ref{canonical}), we obtain canonical strict $*$-pseudofunctors $\ff$ and $\gg$ associated to $\Theta[S]$ and $\Theta[T],$ respectively.
 
 For the remaining assertions, let's first assume that $\ff$ and $\gg$ are unitarily equivalent via the the unitary pseudonatural transformation $\theta: \ff\Rightarrow \gg.$ This provides us a family of unitaries $\theta_e \in 2\ucat(\Theta(\hh^e) \Rightarrow\Theta(\tilde{\hh}^e)).$ Therefore,  by Proposition \ref{funda1}, for each $e\in E(\Gamma),$ we obtain the relation $\Theta(C^e) = (\theta_e\otimes\theta_{\bar{e}})\circ(\Theta(\tilde{C}^e)).$ Thus, defining $U^e:= \Theta^{-1}(\theta_e),$ we obtain the desired family of unitaries in $\hilb$ witnessing the desired equivalence. 
 
For the reversed direction, notice that the hypotheses in the converse of Proposition \ref{funda1} are met via the given family consisting of unitaries $U^e.$ This provides the desired modifications and unitary pseudonatural isomorphisms. The proof is therefore completed.
\end{proof}

\begin{remark}
Observe that in the previous corollary we asked for the indexing sets $S$ and $T$ to be identical. However, this need not always be the case. Say we have $S = (J, H, C)$ and $T = (\tilde J, \tilde H, \tilde C)$ that determine equivalent unitary modules. One can prove that that for each $a\in V(\Gamma)$, there exists a bijection $\varphi^a: J^a\rightarrow\tilde {J}^a.$ We can the introduce $\varphi^{-1}[T] := ( \{J^a\}= \{(\varphi^a)^{-1}[\tilde{J}^a]\}_{a\in V(\Gamma)}, \{\tilde{\hh}_{\varphi^a(v)\varphi^b(w)}\} ,\{C^e_{\varphi^a(v)\varphi^b(w)}\})$ and observe this is still a $\Gamma$-fundamental solution in $\hilb$. By doing this, we managed to switch to matching indexing sets for both $S$ and $\varphi^{-1}[T],$ disregarding relabeling of such sets.

\end{remark}

In the remaining part of this section, we introduce yet another technique describing equivalence of unitary modules in terms of antilinear maps between Hilbert Spaces:

$$\Phi^e_{vw}: \hh^e_{vw} \rightarrow \hh^{\overline{e}}_{wv}$$
defined by
\begin{align}\label{phi}
\Phi^e_{vw} (\xi) &:= (\xi^* \otimes \ONE_{\hh^{\overline{e}}_{wv}})(C^e_{vw}(1)),
\end{align}
where $\xi^* := \langle \, \cdot \, , \xi \rangle$.

We now restate the equivalence of unitary modules in terms of these associated antilinear operators.

\begin{proposition} \label{op-eq} Consider two $\Gamma$-fundamental solutions $S = (J,H,C)$ and $T = (\tilde{J},\tilde{H},\tilde{C})$ in $\hilb$ such that for each $a\in V(\Gamma)$ we have $\hh^{\ONE_a} = \tilde{\hh}^{\ONE_a}$ and $C^{\ONE_a} = \tilde{C}^{\ONE_a},$ with associated anti-linear maps $\set{\Phi^e_{vw}}$, $\set{\Psi^e_{vw}},$ respectively. Moreover, let $\ff, \gg$ be the unitary modules associated with $\Theta[S]$ and $\Theta[T],$ respectively. 
Then $\ff$ and $\gg$ are unitarily equivalent if and only if for every vertex $a \in V(\Gamma)$ there exists a bijection $\varphi^a: J^a \rightarrow \tilde{J}^a$ and every edge $e: a \rightarrow b$ there exists a unitary  $$U^e_{vw}: \tilde{\hh}^{e}_{\varphi^{a}(v)\varphi^b(w)} \rightarrow \hh^e_{vw}$$
such that
\begin{align}
    \Phi^e_{vw} &= U^{\overline{e}}_{wv}\circ \Psi^{e}_{\varphi^a(v)\varphi^b(w)}\circ (U^{e}_{vw})^*.
\end{align}
In other words, there exist unitaries such that the following diagram commutes for every $e \in E(\Gamma)$
\begin{equation}
  \begin{tikzcd}[every arrow/.append style={shift left}, column sep=large, row sep=large]
    \tilde{\hh}^{e}_{\varphi^a(v)\varphi^b(w)} \arrow{r}{\Psi^{e}_{\varphi^a(v)\varphi^b(w)}} \arrow{d}[left]{U^e_{vw}} & \tilde{\hh}^{\overline{e}}_{\varphi^a(w)\varphi^b(v)} \arrow{d}{U^{\overline{e}}_{wv}} \\
    \hh^e_{vw} \arrow{r}[below]{\Phi^e_{wv}} & \hh^{\overline{e}}_{wv} 
  \end{tikzcd}
\end{equation}
\end{proposition}
\begin{proof}
First assume that $\ff$ and $\gg$ are unitarily equivalent via the pseudonatural isomorphism $\theta.$ For each edge $e\in E(\Gamma)$, we have unitary 2-morphisms in $2\hilb$ given by $\Theta^{-1}(\theta_e):\hh^e\rightarrow\tilde{\hh}^e.$ (We remind the reader that the notation $2\hilb$ is used to simply denote a 2-morphism space $\hilb,$ without specifying the underlying 1-morphisms, as introduced in Notation \ref{hom notation}) Thus, for each $a\in V(\Gamma),$ and each $v\in J^a$ there exists a unique $\varphi^a(v)\in \tilde{J}^a$ such that $U^e_{vw} :=[\Theta^{-1}(\theta_e)]_{vw}: \hh^e_{vw} \rightarrow \tilde{\hh}^e_{\varphi^a(v)\varphi^b(w)}$ defines a unitary between Hilbert spaces. Moreover, the correspondences $\varphi^a$ are necessarily bijective, since $\Theta^{-1}(\theta_e)$ is a unitary isomorphism between bigraded Hilbert spaces. Furthermore, by Corollary \ref{funda2}, it follows that $C^e_{vw} = (U^e_{vw}\otimes U^{\bar{e}}_{wv})\circ \tilde{C}^{e}_{\varphi^a(v)\varphi^b(w)}$.
Now, for any vector $\xi$, by expanding each $C^e_{vw}(1)$ in an arbitrary chosen basis for each Hilbert space, we obtain the following chain of equalities:
\begin{align*}
    \Phi^e_{vw}\circ U^e_{vw}(\xi) &= (\langle\cdot,U^e_{vw}(\xi)\rangle\otimes \ONE)\big(C^e_{vw}(1)\big)\\
    &= \big[\langle\cdot,U^e_{vw}(\xi)\rangle\otimes \ONE\big]\bigg(\sum_i a_i\otimes b_i\bigg)\\
    &= \sum_i \langle a_i,U^e_{vw}(\xi)\rangle\cdot b_i\\
    &= \sum_i U^{\bar {e}}_{wv}\big[\langle\cdot,\xi\rangle\otimes \ONE\big]\big[((U^e_{vw})^*a_i)\otimes ((U^{\bar{e}_{wv}})^*b_i)\big]\\
    &= U^{\bar {e}}_{wv}\big[(\langle\cdot,\xi\rangle\otimes \ONE)\big][(U^e_{vw})^*\otimes (U^{\bar{e}}_{wv})^*]\bigg(\sum_i a_i\otimes b_i\bigg)\\
    &= U^{\bar{e}}_{wv}\big[(\langle\cdot,\xi\rangle\otimes \ONE)\big]\big(\tilde{C}^{e}_{vw}(1)\big)\\
    &= U^{\bar{e}}_{wv}\circ\Psi^e_{vw}(\xi).
\end{align*}
This proves the forward direction.

The converse follows by computation, by choosing orthonormal bases for each Hilbert space and concluding by applying the converse of Corollary \ref{funda2}.
\end{proof}

By similar arguments as those found in \cite{DeCommer2013TannakaSU2}, we find the following necessary and sufficient conditions in order for families of anti-linear maps to be associated to fundamental solutions, allowing us to pass back and forth between these two. 

\begin{proposition} \label{op_prop}
Suppose we have a $\Gamma$-fundamental solution $S = (J, \hh, C)$ and $\{\Phi\}$ as in Equation (\ref{phi}). Then the family of operators 
$\set{\Phi^e_{vw}}^{e\in E(\Gamma)}_{vw}$ satisfy 
\begin{equation}\label{zig}
    \Phi^{\overline{e}}_{wv}\Phi^e_{vw} = \ONE_{\hh^{e}_{vw}}\ \text{and}\\
\end{equation}
\begin{equation}\label{zag}
    \sum_{w \in J^b} \Tr\left((\Phi^e_{vw})^*\Phi^e_{vw}\right) = \delta_e \text{, for every } v \in J^a.
\end{equation}
    Conversely, if a collection of antilinear operators $\set{\Phi^e_{vw}}$ satisfy these conditions, then the family $\set{C^e_{vw}}$ defined by $C^{e}_{vw}(1) := \sum_i \xi_i\otimes\Phi(\xi_i)$,  satisfy the zig-zag relations (Definition \ref{fundamentalsolution}), where $\{\xi_i\}$ is an ONB. (We remark that the definition of $C^e_{vw}$ is independent of the choice of ONB $\{\xi_i\}$.)
\end{proposition}
\begin{proof} Let us first check that \eqref{zig} holds. Unwinding the definition of $\Phi^e_{vw}$, we have for any $\xi \in \hh^e_{vw}$
$$\Phi^{\overline{e}}_{wv}\Phi^e_{vw}(\xi) = \Phi^{\overline{e}}_{wv}\left[(\xi^* \otimes \ONE_{\hh^{e}})(C^e_{vw}(1))\right] =  \left((C^{\overline{e}}_{wv})^* \otimes \ONE_{\hh^{\overline{e}}}\right)\left(\xi \otimes C^e_{vw}(1)\right).$$
Using the first equality in Example \ref{solinhilb}, we obtain that the right hand side is equal to $\ONE_{\hh^e_{vw}}$. We now verify equation \eqref{zag}. First choose orthonormal bases $(\xi_i)_i$ of the Hilbert spaces $\hh^e_{vw}$. Then 
\begin{align*}
    \sum_w (C^e_{vw})^* C^e_{vw} &= \sum_w \sum_i \left((\xi^{*}_i \otimes \ONE_{\hh^{\overline{e}}}) C^e_{vw}\right)^* \left((\xi^{*}_i \otimes \ONE_{\hh^{\overline{e}}}) C^e_{vw}\right)\\
    &= \sum_w \sum_i \langle \Phi^e_{vw} \xi_i, \Phi^e_{vw} \xi_i \rangle\\
    &= \sum_w \Tr((\Phi^e_{vw})^* \Phi^e_{vw}).
\end{align*}
By the second equality in Example \ref{solinhilb}, we have that $\sum_w \Tr((\Phi^e_{vw})^* \Phi^e_{vw}) = \delta_e$ for every $v$. The converse holds by similar arguments, taking each $C^e_{vw}$ to be the unique map such that Equation (\ref{phi}) holds for $\Phi^e_{vw}$.
\end{proof}

\begin{remark}
The above conditions imply that $\hh^e_{vw}$ and $\hh^{\overline{e}}_{wv}$ have the same dimension for every $v,w$ and edge $e \in E(\Gamma)$, since $\Phi^e_{vw}$ is invertible.
\end{remark}

\section{Classification of Unitary TLJ-modules by Graphs} \label{last_section}
We use the equivalences from the previous section to classify certain unitary modules of graph-generated Temperley-Lieb categories in terms of edge-colored oriented weighted graphs. We first introduce notation and basic definitions. Throughout this section we let $\Gamma$ be a fixed but arbitrary weighted bidirected graph, as in Definition \ref{bidirectedgraph}.  We reserve the symbols $S = (J,H,C)$ for a $\Gamma$-fundamental solution in $\hilb$. Furthermore, we denote the associated anti-linear maps of $S$ by $\set{\Phi^e_{vw}}$ as defined in \eqref{phi}. We also reserve $e$ for edges in $\Gamma$ and $\epsilon$ for edges in the graphs we will use to classify our unitary modules.

\begin{notation}
Let $(\lambda_k^{(e,vw)})_k$ denote the eigenvalues of the bounded linear transformation $[(\Phi^e_{vw})^* \Phi^e_{vw}]: \hh^e_{vw}\rightarrow \hh^e_{vw}$ counted with multiplicity.
\end{notation}

Now we construct a weighted oriented graph using the spectral data of these operators.

\begin{definition}
We define the graph $(\Lambda_S,w_S,\pi_S)$ generated by a $\Gamma$-fundamental solution $S$ in $\hilb$ as the weighted oriented graph, which has the vertex set $V(\Lambda_S) := \sqcup J$ (the disjoint union of the sets $J^a\in J$ produces the \textbf{indexed} collection of all points in the sets in the collection $J$) and for each edge $e:a \rightarrow b$ in $\Gamma$ we trace $\text{dim}(\hh^e_{vw})$ arrows $\{\epsilon^{(e)}_k\}_k$ from $v \in J^a$ to $w \in J^b$ with weights $(\lambda^{(e,vw)}_k)_k$ given by the spectrum of $[(\Phi^e_{vw})^*\Phi^e_{vw}]$, counted with multiplicity. Notice this specifies a weight function $w_S: E(\Lambda_S) \rightarrow (0,\infty).$ We then define $\pi_S: \Lambda_S \rightarrow \Gamma$ as the graph homomorphism that sends every $v \in J^a$ to $a$ and every $\epsilon^{(e)}_k$ to $e$.
\end{definition}
We now show a simple example to explain the relevance of the disjoint union in the previous paragraph. Say, $J = \{J^a = \{v,w\}, J^b = \{w,z\}\}.$ We then have that $\sqcup J = \{v_a,w_a,w_b, z_b\}\neq \{v,w,z\} = \cup J.$ Notice how the advantage of taking a disjoint union is that it ``remembers" where every element came from.
\begin{definition}\label{fairgraph}
We say a weighted directed graph $(\Lambda, w,\pi)$ with a graph homomorphism $\pi: \Lambda \rightarrow \Gamma$ is a $\Gamma$-\textbf{fair} graph if and only if 
for each $e:a\rightarrow b\in E(\Gamma)$ and every vertex $\alpha \in \pi^{-1}(a)$
$$\sum_{\substack{source(\epsilon) = \alpha \\ \pi(\epsilon) = e}} w(\epsilon) = \delta_e.$$
\end{definition}
\begin{remark}
We observe that if $(\Lambda,w,\pi)$ is a $\Gamma$-fair graph, then necessarily $\pi$ is surjective onto $E(\Gamma)$, as otherwise the summation condition in Definition \ref{fairgraph} would give an edge $e\in E(\Gamma)$ with $\delta_e = 0,$ contradicting the initial assumption that $\Gamma$ is a weighted bi-directed graph.
\end{remark}

\begin{definition}\label{zen}
We say two $\Gamma$-fair graphs $(\Lambda_1, w_1, \pi_1)$, $(\Lambda_2, w_2, \pi_2)$ are isomorphic if and only if there exists a graph isomorphism $\varphi: \Lambda_1 \rightarrow \Lambda_2$ such that $\pi_1 = \pi_2 \circ \varphi$ and $w_1 = w_2 \circ \varphi$.
\end{definition}
\begin{definition} \label{balanced}
We say a $\Gamma$-fair graph $(\Lambda,w,\pi)$ is \textbf{balanced} if and only if there exists an involution (\,$\overline{\, \cdot \,}$\,) on $E(\Lambda)$ that switches sources and targets, such that for every $\epsilon \in E(\Lambda)$
\begin{align*}
    w(\epsilon)w(\overline{\epsilon}) &= 1,\ and\\
    \pi(\overline{\epsilon}) &= \overline{\pi(\epsilon)}.
\end{align*}
\ \\
Note that the involution on the left hand side of the last equation is that of $\Lambda$, and the involution on the right hand side is the involution on $\Gamma$. We conclude this Definition by remarking that the existence of such an involution is a property and not extra structure, as in (\cite{DeCommer2013TannakaSU2}, p2 Remark 1).
\end{definition}
We provide an example of a balanced $\Gamma$-fair graph for a chosen bi-directed graph $\Gamma$.
\begin{example}
Let $\Gamma_1$ be the following weighted bidirected graph.

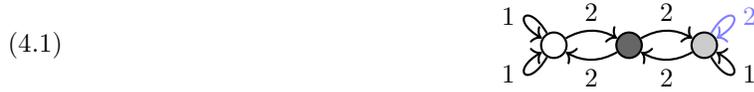
\begin{equation}
    \begin{tikzpicture}[thick,baseline={([yshift=-\the\dimexpr\fontdimen22\textfont2\relax] current bounding box.center)}]
        \node (a) at ( 2,0) [circle,draw=black!100,fill=black!20] {};
        \node (b) at ( 1,0) [circle,draw=black!100,fill=black!60] {};
        \node (c) at ( 0,0) [circle,draw=black!100] {};
        \draw[->] (a) to[bend left] node[below] {$2$} (b);
        \draw[->] (b) to[bend left] node[above] {$2$} (a);
        \draw[->] (c) to [in=150,out=120,loop] node[left] {$1$} (c);
        \draw[->] (c) to [in=210,out=240,loop] node[left] {$1$} (c);
        \draw[->,blue!50] (a) to [in=30,out=60,loop] node[right] {$2$} (a);
        \draw[->] (a) to [in=-30,out=-60,loop] node[right] {$1$} (a);
        \draw[->] (c) to[bend left] node[above] {$2$} (b);
        \draw[->] (b) to[bend left] node[below] {$2$} (c);
    \end{tikzpicture}
\end{equation}
Then the weighted bidirected graph $\Lambda_1$ shown below is a balanced $\Gamma_1$-fair graph.

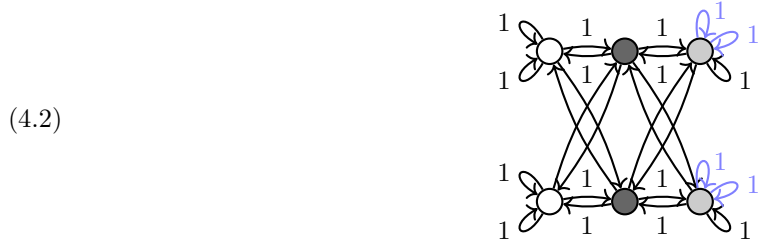
\begin{equation}
    \begin{tikzpicture}[thick,baseline={([yshift=-\the\dimexpr\fontdimen22\textfont2\relax] current bounding box.center)}]
        \node (c)  at (-1,1)  [circle,draw=black!100]                {};
        \node (c2) at (-1,-1) [circle,draw=black!100] {};
        \node (a1) at (0,1)  [circle,draw=black!100,fill=black!60]   {};
        \node (a2) at (0,-1) [circle,draw=black!100,fill=black!60]   {};
        \node (b1) at (1,1)   [circle,draw=black!100,fill=black!20] {};
        \node (b2) at (1,-1)  [circle,draw=black!100,fill=black!20] {}; 
        \draw[->] (c) to [in=150,out=120,loop] node[left] {$1$} (c);
        \draw[->] (c) to [in=210,out=240,loop] node[left] {$1$} (c);
        \draw[->] (c) to[bend left=10] node[above] {$1$} (a1);
        \draw[->] (c) to[bend left=10] node [above right] {} (a2);   
        \draw[->] (a1) to[bend left=10] node[below] {$1$} (c);
        \draw[->] (a2) to[bend left=10] node[below] {} (c);   
        \draw[->] (c2) to [in=150,out=120,loop] node[left] {$1$} (c2);
        \draw[->] (c2) to [in=210,out=240,loop] node[left] {$1$} (c2);
        \draw[->] (c2) to[bend left=10] node[above] {$1$} (a2);
        \draw[->] (c2) to[bend left=10] node[above right] {} (a1);   
        \draw[->] (a1) to[bend left=10] node[below right] {} (c2);
        \draw[->] (a2) to[bend left=10] node[below] {$1$} (c2);   
        \draw[->] (a1) to[bend left=10] node[above] {$1$} (b1);
        \draw[->] (b1) to[bend left=10] node[below] {$1$} (a1); 
        \draw[->] (a1) to[bend left=10] node[above] {} (b2);
        \draw[->] (b2) to[bend left=10] node[below] {} (a1); 
        \draw[->] (a2) to[bend left=10] node[above] {$1$} (b2);
        \draw[->] (b2) to[bend left=10] node[below] {$1$} (a2); 
        \draw[->] (a2) to[bend left=10] node[above] {} (b1);
        \draw[->] (b1) to[bend left=10] node[below] {} (a2);  
        \draw[->] (b1) to [in=-30,out=-60,loop] node[right] {$1$} (b1);
        \draw[->] (b2) to [in=-30,out=-60,loop] node[right] {$1$} (b2);
        \draw[->,blue!50] (b1) to [in=10,out=40,loop] node[right] {$1$} (b1);
        \draw[->,blue!50] (b1) to [in=70,out=100,loop] node[right] {$1$} (b1);
        \draw[->,blue!50] (b2) to [in=10,out=40,loop] node[right] {$1$} (b2);
        \draw[->,blue!50] (b2) to [in=70,out=100,loop] node[right] {$1$} (b2);
    \end{tikzpicture}
\end{equation}
\end{example}

\begin{remark} In the case where $\Gamma$ has only one vertex and one edge, being a balanced $\Gamma$-fair graph is the same as being a fair and balanced $\delta$-graph, as in \cite{DeCommer2013TannakaSU2}.
\end{remark}
\begin{remark}\label{MW condition}
At this stage, it is important to mention the graphs studied in \cite{Morrison2010Graphpre-print}, which we denote by \textbf{MW-type} graphs. Consider a graph homomorphism $\pi: \Lambda\rightarrow \Gamma$ onto $E(\Lambda),$ where $\Lambda$ comes equipped with a \textit{Perron-Frobenius dimension data for} $\pi$, $d:V(\Lambda)\rightarrow (0,\infty),$ satisfying the following two conditions: for every $(e:a\rightarrow b)\in E(\Gamma)$ we have that
$$\delta_{e} = \sum_{(\alpha\rightarrow \beta) \,\in\, \pi^{-1}[e]} \frac{d(\beta)}{d(\alpha)},$$
for each $\alpha \in \pi^{-1}[a],$ and
$$\delta_{e} = \sum_{(\alpha \rightarrow \beta) \,\in\, \pi^{-1}[e]} \frac{d(\alpha)}{d(\beta)}$$
for each $\beta \in \pi^{-1}[b]$.

It is easy to see that MW-type graphs together with all the information listed above constitute examples of balanced $\Gamma$-fair graphs. However these conditions are not exactly equivalent as we will see in the following proposition, which is based on the discussion on top of page 12 of \cite{Hartglass2018Non-tracialAlgebras}.
\end{remark}
\begin{proposition}
Let $(\Lambda,w,\pi)$ be a balanced $\Gamma$-fair graph such that for each loop $\alpha_0 \stackrel{\epsilon_0}{\rightarrow} \beta_0 = \alpha_1 \stackrel{\epsilon_1}{\rightarrow} \beta_1 = \alpha_2 \stackrel{\epsilon_2}{\rightarrow}\hdots \stackrel{\epsilon_n}{\rightarrow}\beta_n = \alpha_0$ in $\Lambda$ we have that $\Pi_{i = 0}^{n}w(\epsilon_i) = 1.$ Then $\pi: \Lambda\rightarrow\Gamma$ gives an MW-type graph.
\end{proposition}

\begin{proof}
First define the dimension function $d$ on all of $V(\Lambda).$ Start by fixing an arbitrary vertex $\alpha\in V(\Lambda)$ and defining $d(\alpha) = 1.$ Now if $(\epsilon:\alpha\rightarrow\beta)\in E(\Lambda)$, we simply define $d(\beta) = d(\alpha)\cdot w(\epsilon).$ We shall then show that we can extend this function to any arbitrary vertex $\beta\in V(\Lambda).$ Let $l = (\epsilon_0,\epsilon_1,\hdots,\epsilon_n)$ be a path in $\Lambda$ starting at $\alpha$ and ending at $\beta$. We then define $d(\beta) := \Pi_{i= 0}^{n}w(\epsilon_i).$ Notice that this indeed yields a well-defined function on $V(\Lambda),$ as made possible by the loop condition stated above; this is, the definition of $d(\beta)$ is independent of the choice of path joining $\alpha$ with $\beta$. It is now immediate that the function $d$ is indeed a Perron-Frobenius dimension function.
\end{proof}

\begin{proposition}\label{is balanced}
Let $S$ be a fundamental solution in $\hilb$. Then the graph $(\Lambda_S,w_S,\pi_S)$ generated by $S$ is a balanced $\Gamma$-fair graph.
\end{proposition}
\begin{proof}
From Proposition \ref{op_prop}, we have for every $v \in J^a$
$$\sum_{\substack{source(\epsilon) \,= \,v\\ \pi_S(\epsilon) \,= \,e}} w(\epsilon) = \sum_w \sum_k \lambda_k^{(e,vw)} = \sum_w \Tr\big((\Phi^e_{vw})^*\Phi^e_{vw}\big) = \delta_e.$$ 
    Moreover, if there are no arrows from $v$ to $w$, then dim$(\hh^e_{vw}) = 0$. By the remark following Proposition \ref{op_prop} we conclude that dim$(\hh^{\overline{e}}_{wv}) = 0$ as well, so there are no edges from $w$ to  $v$ either.  Assume now that we are not in this trivial case. We consider the left polar decomposition of the maps $\Phi^e_{vw} = V^e_{vw} |\Phi^e_{vw}|$ and so that
$$V^e_{vw}: \hh^e_{vw} \rightarrow \hh^{\overline{e}}_{wv}$$
are isometric anti-linear maps and
$$|\Phi^e_{vw}|: \hh^e_{vw} \rightarrow \hh^e_{vw}$$
are positive linear maps. From $\text{dim}(\hh^e_{vw}) = \text{dim}(\hh^{\overline{e}}_{wv})$, we know that $V^e_{vw}$ is an anti-unitary since it is an anti-linear isometry. Now from the first equation in Proposition \ref{op_prop}, it follows that
\begin{align*}
    \Phi^{\overline{e}}_{wv} &= (\Phi^e_{vw})^{-1}\\
    &= |\Phi^e_{vw}|^{-1} (V^e_{vw})^{-1}\\
    &= \underbrace{(V^e_{vw})^*}_{\text{anti-unitary}} \underbrace{V^e_{vw} |\Phi^e_{vw}|^{-1} (V^e_{vw})^*}_{\text{positive}}
\end{align*}
By uniqueness of left polar decomposition we obtain that
\begin{align*}
U^{\overline{e}}_{wv} = (U^e_{vw})^*,\\
|\Phi^{\overline{e}}_{wv}| = U^e_{vw} |\Phi^e_{vw}|^{-1} (U^e_{vw})^*.
\end{align*}
Let us consider the spectrum of $(\Phi^e_{vw})^*\Phi^e_{vw}$ counted with multiplicity. We find that
\begin{align*}
    \sigma\big((\Phi^e_{vw})^*\Phi^e_{vw}\big) &= \sigma\Big(\big((\Phi^e_{vw})^*\Phi^e_{vw}\big)^{-1} \Big)^{-1}\\
    &= \sigma\Big( \big(|\Phi^e_{vw}|\big)^{-1} \big(|\Phi^e_{vw}|^* \big)^{-1}\Big)^{-1}\\
    &= \sigma\big(  (\Phi^{\overline{e}}_{wv})^* \Phi^{\overline{e}}_{wv} \big)^{-1}
\end{align*}
by using the relations between the polar decompositions of $\Phi^{e}_{vw}$ and $\Phi^{\overline{e}}_{wv}$ found above. First, note that each term above is well-defined, as these are invertible operators. Second, notice that our use of the sprectral theorem is justified, as we are dealing with bounded self-adjoint operators. Thus, for every edge $\epsilon: v \rightarrow w$ in $\Lambda_S$ with $\pi(\epsilon) = e$, there exists another edge $\epsilon': w \rightarrow v$ with the property that $\pi(\epsilon') = \overline{e}$ such that $w(\epsilon)w(\epsilon') = 1$. Hence $\Lambda_S$ is a balanced $\Gamma$-fair graph.
\end{proof}

\begin{definition}
Given a balanced $\Gamma$-fair graph $(\Lambda, w, \pi)$, we generate a fundamental solution $S_\Lambda$ in $\hilb$ as follows: 
\begin{itemize}
    \item Take $J^a := \pi^{-1}(a)$ for every $a \in V(\Gamma)$.
    \item We define $\hh_{vw}^e := \C[\set{(\epsilon: v \rightarrow w)\in E(\Lambda) \mid \pi(\epsilon) = e}]$ taking the formal complex linear span. By regarding the edges as an orthonormal basis, $\hh_{vw}^e$ is then turned into a Hilbert space. Now take $\hh^e := \oplus_{vw} \hh_{vw}^e$. We remark that since $\Lambda$ is a balanced $\Gamma$-fair graph, $\hh^e$ must be row and column finite. Thus $\hh^e \in \text{Hilb}^{J^a \times J^b}_f$.
    \item Let $\overline{\, \cdot \,}$ be a fixed but arbitrary involution on $E(\Lambda)$, satisfying the conditions in Definition \ref{balanced}, whose existence is guaranteed by hypothesis. Notice that this involution naturally extends to a well-defined anti-linear map $\overline{\, \cdot \,}: \hh^e_{vw}\rightarrow \hh^{\bar{e}}_{wv}$, for which we keep the same notation. We similarly take $\Phi_{vw}^e: \hh^e_{vw} \rightarrow \hh^{\overline{e}}_{wv}$ as the unique anti-linear map, defined on the standard basis vectors as $\epsilon \mapsto w(\epsilon)^{1/2} \, \overline{\epsilon}$. Now we find
    \begin{align*}
        \Phi^{\overline{e}}_{wv}\Phi^e_{vw} \epsilon &= \Phi^{\overline{e}}_{wv} w(\epsilon)^{\frac{1}{2}} \, \overline{\epsilon}\\
        &= w(\overline{e})^{\frac{1}{2}} w(\epsilon)^{\frac{1}{2}} \epsilon\\
        &= \epsilon. 
    \end{align*}
    Hence $\Phi^{\overline{e}}_{wv}\Phi^e_{vw} = \ONE_{\hh^e}$. Furthermore, for each edge $e\in E(\Gamma)$ we have that
    \begin{align*}
    \sum_{w \,\in\, J^b} \Tr\big((\Phi^e_{vw})^* \Phi^e_{vw}\big) = \sum_{\substack{source(\epsilon) \,=\, v \\ \pi(\epsilon) \,=\, e}} w(\epsilon) = \delta_e.
    \end{align*}
    It then follows by Proposition (\ref{op_prop}) that the family $\set{\Phi^e_{vw}}$ uniquely define $\set{C^e_{vw}}$ that satisfy the zig-zag relations.
\end{itemize}
\end{definition}

\begin{remark}\label{involutions}
We remark that if we have two balanced involutions $(\,\overline{\, \cdot \,}^{\,_1}\,)$ and $(\,\overline{\, \cdot \,}^{\,_2}\,)$ on a given $\Gamma$-fair and balanced graph $(\Lambda,w,\pi),$ the associated $\Gamma$-fundamental solutions to families of anti-linear maps, $\{\Phi^e_{vw}\}$ and $\{\Psi^e_{vw}\}$ define isomorphic canonical strict $*$-pseudofunctors from TLJ$(\Gamma)$ into $\ucat$. To see this it suffices to verify it on the (basis) edges. Consider the associated Hilbert space $\hh^e_{vw}.$ If we assume $e = \bar{e}$ and $v = w$, we need to construct a unitary $U^e_{vv} : \hh^e_{vv} \rightarrow \hh^e_{vv}$ such that $\Phi_{vv}^{e} = U^e_{vv}\circ\Psi_{vv}^{e}\circ (U^e_{vv})^*.$ Notice how if $\{\epsilon_i :\alpha\rightarrow\alpha\}_{i = 1}^{M}$ are all the loops in $\Lambda$ coming out of $\alpha\in V(\Lambda)$ projecting onto $e$ in $\Gamma$, we can re-enumerate them starting by the fixed edges ($\overline{\,\epsilon_{i} \,}^{\,_1} = \epsilon_{i}$) as $\{\epsilon_{i_0}\}_{i_0\in I_0},$ and the remaining edges in such a way that $\overline{\, \epsilon_{2i-1} \,}^{\,_1} = \epsilon_{2i}.$  We can therefore express $\{\epsilon_i\}_{i = 1}^{M} = \{\epsilon_{i_0}\}_{i_0\in I_0}\bigsqcup\{\epsilon_{i_1}\}_{i_1\in I_1} \bigsqcup \{ \epsilon_{i_x}\}_{i_x\in I_x}\bigsqcup \hdots \bigsqcup \{ \epsilon_{i_z}\}_{i_z\in I_z},$ corresponding to the fibers of the weight. Here $1 < x < \hdots < z,$ and moreover, $I_0$ denotes the edges fixed by $(\overline{\,\cdot\,}^{\,_1}).$ Notice then that both involutions simply permute these sets, respecting the partition by weights. We therefore express our involutions as the disjoint product of transpositions:\\

$(\,\overline{\, \cdot \,}^{\,_1}\,) : \
\underbrace{(1)(2)\hdots (n_1-1)}_{ \text{fixed points (weight  1)}}\cdot\underbrace{(n_1\ n_1+1)\hdots (n_x-2\ n_x-1)}_{\text{weight}\ 1 }\cdot $\\ 

\hspace{26 mm}$\cdot \underbrace{(n_x\ n_x+1)(n_x+2\ n_x+3)\hdots(n_y-2\ n_y-1)}_{\text{weight $x$ or $1/x$}}\cdot\hdots\cdot \underbrace{(n_z\ n_z+1)\hdots (M-1\ M)}_{\text{weight $z$ or $1/z$}},$
\ \\

and $(\,\overline{\, \cdot \,}^{\,_2}\,)$ in the symbols $\xi_k$, expressed as

\ \\

$(\,\overline{\, \cdot \,}^{\,_2}\,) : \
\underbrace{(\xi_1)(\xi_2)\hdots (\xi_{n_1-1})}_{ \text{fixed points (weight 1)}}\cdot\underbrace{(\xi_{n_1}\ \xi_{n_1+1})\hdots (\xi_{n_x-2}\ \xi_{n_x-1})}_{\text{weight}\ 1 }\cdot$ \\

\hspace{23 mm}$\cdot \underbrace{(\xi_{n_x}\ \xi_{n_x+1})(\xi_{n_x+2}\ \xi_{n_x+3})\hdots(\xi_{n_y-2}\ \xi_{n_y-1})}_{\text{weight $x$ or $1/x$}}\cdot\hdots\cdot \underbrace{(\xi_{n_z}\ \xi_{n_z+1})\hdots (\xi_{M-1}\ \xi_M)}_{\text{weight $z$ or $1/z$}}.$\\
For each weight $x$ with $1 < x$, we denote by $g_x$ the uniquely determined permutation such that $$g_x(\xi_{n_x}\ \xi_{n_x+1})\hdots(\xi_{n_y-2}\ \xi_{n_y-1})g_x^{-1} = (n_x\ n_x+1)\hdots(n_y-2\ n_y-1).$$ We are now ready to describe $U^e_{vv}$ in terms of its action on this ordered basis: for the basis elements whose weight is given by $x > 1$, we simply define $U^e_{vv}$ to act as the corresponding permutation $g_x$. We are now only left with edges whose weight is one. By observing that if an expression of the form $(\xi)(\gamma)$ appears in either involution, it can be made unitarily equivalent to the involution containing all the same fixed points and transpositions, but containing $(\xi\ \gamma)$. We thus define the action of $U^e_{vv}$ on the subspace these edges generates is described by first applying the unitary matrix:
\begin{center}$
    \begin{bmatrix}
    1/\sqrt{2} & 1/\sqrt{2}\\
    i/\sqrt{2} & -i/\sqrt{2}
    \end{bmatrix},$
\end{center}
followed by the permutation switching the corresponding symbols from one involution to the other. If $\xi$ remains fixed by both involutions, we let $U^e_{vv}$ act trivially on $\xi$. This fully determines $U^e_{vv}$, as we described its action on a basis.

In any other case, whenever $e\neq \bar{e}$ or $v\neq w,$ we find that for each duality pair $\{e, \bar{e}\}$ and each pair of vertices $\{v,w\}$ we have that
$$\Phi^e_{vw}(\epsilon) = \bigg[(\overline{\, \, \overline{\,\cdot\,}^{\,_2}}^{\,_1}) \circ \Psi^e_{vw}\circ \ONE_{\hh^e_{vw}} \bigg](\epsilon), \text{ and}$$

$$\Phi^{\overline{e}}_{wv}(\epsilon) = \bigg[ \ONE_{\hh^e_{vw}} \circ \Psi^{\overline{e}}_{wv}\circ (\overline{\, \, \overline{\,\cdot\,}^{\,_2}}^{\,_1})^* \bigg](\epsilon) = \bigg[ \ONE_{\hh^e_{vw}} \circ \Psi^{\overline{e}}_{wv}\circ (\overline{\, \, \overline{\,\cdot\,}^{\,_1}}^{\,_2}) \bigg](\epsilon).$$
Here, $ U^{\overline{e}}_{wv} := (\overline{\, \, \overline{\,\cdot\,}^{\,_2}}^{\,_1})$ and $U^e_{vw} := \ONE_{\hh^e_{vw}}$.

These cases provide explicit unitaries witnessing the equivalence of $\Gamma$-fundamental solutions. Finally, to be able to use Proposition \ref{op-eq}, we need to verify that for each $a\in V(\Gamma)$ and each pair $v,w\in J^a$ we have that $C^{\ONE_a}_{vw} = \tilde{C}^{\ONE_a}_{vw}.$ However, one can see this by computation, using the unitaries described above and thus completing the proof.
\end{remark}

The following theorem further reduces the equivalence of $*$-pseudofunctors $\ff: \TLJ\rightarrow \hilb$ in terms of balanced $\Gamma$-fair graphs:

\begin{theorem}
 When $S,T$ are $\Gamma$-fundamental solutions in $\hilb$ with $\hh^{\ONE_a} = \tilde{\hh}^{\ONE_a}$ and $C^{\ONE_a}= \tilde{C}^{\ONE_a}$ for each vertex $a\in V(\Gamma)$, then the associated $\Gamma$-fair graphs $(\Lambda_S,w_S,\pi_S)$ and $(\Lambda_T,w_T,\pi_T)$ are isomorphic as $\Gamma$-fair graphs if and only if the strict $*$-pseudofunctors $\TLJ\rightarrow \ucat$ induced by $\Theta[S]$ and $\Theta[T]$ are unitarily equivalent.
\end{theorem}

\begin{proof}
From Proposition \ref{op-eq}, two fundamental solutions $S,T$ with associated anti-linear maps $\set{\Phi^e_{vw}}$ and $\set{\Psi^e_{vw}}$, respectively, induce unitarily equivalent $*$-pseudofunctors if and only if for every vertex  $a \in V(\Gamma)$ there exists a bijection $\varphi^a: J^a \rightarrow \tilde{J}^a$ and every edge $e: a \rightarrow b$ there exists a unitary
$$U^e_{vw}: \tilde{\hh}^{e}_{\varphi^a(v)\varphi^b(w)} \rightarrow \hh^e_{vw}$$
such that
$$\Phi^e_{vw} = U^{\overline{e}}_{wv} \Psi^{e}_{\varphi^a(v)\varphi^b(w)} (U^{e}_{vw})^*.$$
Observe that the collection of bijections $\set{\varphi^a}_{a \in V(\Gamma)}$, induces an obvious bijection between $V(\Lambda_S)$ and $V(\Lambda_T)$. Furthermore,
\begin{align*}
    \sigma\Big((\Phi^e_{vw})^* \Phi^e_{vw} \Big) =  \sigma\Big( U^{e}_{vw} (\Psi^e_{\varphi^a(v)\varphi^b(w)})^* \Psi^e_{\varphi^a(v)\varphi^b(w)} (U^{e}_{vw})^* \Big) = \sigma\Big((\Psi^e_{\varphi^a(v)\varphi^b(w)})^* \Psi^e_{\varphi^a(v)\varphi^b(w)} \Big)
\end{align*}
It then follows that $\Lambda_S$ and $\Lambda_T$ are isomorphic since there exists a graph isomorphism $\varphi: \Lambda_S \rightarrow \Lambda_T$ with $\pi_S = \pi_T\circ \varphi$ and $w_S = w_T\circ \varphi.$

We shall now prove the forward direction. If $\Lambda_1$ and $\Lambda_2$ are isomorphic as balanced $\Gamma$-fair graphs, then there exists a graph isomorphism $\varphi: \Lambda_1 \rightarrow \Lambda_2$ intertwining the data from these graphs. Now consider the fundamental solutions $S_{\Lambda_1}$ and $S_{\Lambda_2}$ generated by $\Lambda_1$ and $\Lambda_2,$ respectively. By restricting $\varphi$ to $J^a := \pi^{-1}_1(a)$ we obtain bijections between $J^a$ and $\tilde{J}^a$, since $\pi_1 = \pi_2 \circ \varphi$. 
Furthermore, consider the maps $U^e_{vw}$ to be the (unitary) linear extension of $\varphi: E(\Lambda_1) \rightarrow E(\Lambda_2)$, as restricted to the corresponding vertices. Thus defining unitaries $U^e_{vw}: \hh^e_{vw} \rightarrow \tilde{\hh}^e_{\phi(v)\phi(w)}$. We now observe that $\varphi^{-1}(\,\overline{\, \varphi(\cdot) \,}^{\,_2}\,) = (U^{e}_{vw})^{*}(\,\overline{\, U^e_{vw}(\cdot) \,}^{\,_2}\,)$ is another balanced $\Gamma$-fair involution on $\Lambda_1$ which is manifestly unitarily equivalent to $(\,\overline{\, \cdot \,}^{\,_2}\,)$. Moreover, by Remark \ref{involutions}, this new involution on $\Lambda_1$ is unitarily equivalent to $(\,\overline{\, \cdot \,}^{\,_1}\,).$ Finally, by Proposition $\ref{op-eq}$, the graphs $S_{\Lambda_1}$ and $S_{\Lambda_2}$ induce unitarily equivalent $*$-pseudofunctors.
\end{proof}

We are now ready to provide a classification of our unitary $\TLJ$-modules.
\begin{theorem}\label{classification}
Every balanced $\Gamma$-fair graph arises from a $\Gamma$-fundamental solution in $\hilb$. Furthermore, there is an equivalence of isomorphism classes of balanced $\Gamma$-fair graphs and unitary isomorphism classes of strong $*$-pseudofunctors $\TLJ \rightarrow \hilb$.
\end{theorem}

\begin{proof}
Let $(\Lambda,w,\pi)$ be a fixed but arbitrary balanced $\Gamma$-fair graph. We shall now construct a fundamental solution $S$ in $\hilb$ such that $\Lambda = \Lambda_S$. For each $a\in V(\Gamma)$, take $J^a := \pi^{-1}(a)$. For $e\in E(\Gamma)$ define $\hh^e_{vw}$ to be vector space spanned by the edges in $\Lambda$ having source $v \in J^{a}$ and range $w \in J^b$ such that $\pi(a\rightarrow b) = e$, and turn it into a Hilbert space by declaring these edges be orthonormal. Then let $\hh^e := \oplus_{vw} \hh^e_{vw}$. Since $\Lambda$ is balanced $\Gamma$-fair, the number of edges coming in or out of any vertex in $\Lambda$ must be uniformly bounded. To see this, we observe first that the sum of the weights of edges in $\pi^{-1}(e)$ must be equal to the sum of their inverses, and second that the $\Gamma$-fair condition imposes that $\pi^{-1}(e)$ is a finite set, as each conjugate adds a weight of at least 1 to that of $e$. Thus, $\hh^e \in \text{Hilb}^{J^a \times J^b}_f$ is a 1-morphism in $\hilb$. 
 Finally, we define $\Phi^e_{vw}: \hh^e_{vw} \rightarrow \hh^{\overline{e}}_{wv}$ as the unique anti-linear map, defined on the standard basis vectors as $\epsilon\mapsto w(\epsilon)^{1/2} \, \overline{\epsilon}$. Here, $(\overline{\, \cdot \,})$ is a fixed but arbitrary involution arising from the balanced hypothesis. Notice that by Remark \ref{involutions}, this definition is independent of the choice of $(\overline{\, \cdot \,})$. It then follows that
\begin{align*}
    \Phi^{\overline{e}}_{wv}\Phi^e_{vw} (\epsilon) &= \Phi^{\overline{e}}_{wv} w(\epsilon)^{\frac{1}{2}} \, (\overline{\epsilon})\\
    &= w(\overline{e})^{\frac{1}{2}} w(\epsilon)^{\frac{1}{2}} \cdot \epsilon\\
    &= \epsilon. 
\end{align*}
Hence $\Phi^{\overline{e}}_{wv}\Phi^e_{vw} = \ONE_{\hh^e_{vw}}$. Furthermore, 
\begin{align*}
\sum_{w \,\in\, J^b} \Tr\big((\Phi^e_{vw})^* \Phi^e_{vw}\big) = \sum_{\substack{source(\epsilon) \,=\, v \\ \pi(\epsilon) \,=\, e}} w(\epsilon) = \delta_e.
\end{align*}
It then follows by Proposition \ref{op_prop} that the family $\set{\Phi^e_{vw}}$ uniquely define $\set{C^e_{vw}}$ that satisfy the zig-zag relations. Therefore the tuple $S = (J,H,C)$ we constructed from $\Lambda$ is a $\Gamma$-fundamental solution in $\hilb$. We now check that $S$ generates $\Lambda$: First notice that $V(\Lambda_S) = V(\Lambda)$ and by how the maps $\set{\Phi^e_{vw}}$ were constructed, $\sigma\big((\Phi^e_{vw})^* \Phi^e_{vw}\big) = \big(w(\epsilon)\big)$ for every $e \in E(\Gamma)$, where both sides are counted with multiplicity. We conclude that $\Lambda_S = \Lambda$.

We shall now show that from a balanced $\Gamma$-fair graph $(\Lambda_S, w_S,\pi_S)$ generated by fundamental solution $S = (J ,\hh, C)$ in $\hilb$, the fundamental solution $T_\Lambda = (\tilde{J},\tilde{\hh},\tilde{C})$ we construct from $(\Lambda_S, w_S,\pi_S)$ is unitarily equivalent to $S$. It is easy to see that $J = \tilde{J}$ and $\hh = \tilde{\hh},$ so it suffices to exhibit the equivalence between $C$ and $\tilde{C}.$ We now endow $\Lambda_S$ with an involution. By (the proof of) Proposition \ref{is balanced}, we know that there exists a balanced $\Gamma$-fair involution on $\Lambda_S$ coming from the spectrum of the maps $\Phi$ associated to $\{C\}$. We denote this involution by $(\,\overline{\, \cdot \,}^{\,_1}\,).$

 Now, to construct the associated linear maps $\Psi$ of $T_\Lambda$, we can make use of any balanced $\Gamma$-fair involution on $\Lambda_S$, denoted by $(\,\overline{\, \cdot \,}^{\,_2}\,).$ However, as explained in Remark \ref{involutions}, we also have that for each $a\in V(\Gamma)$ and each pair $v,w\in J^a$ we have $C^{\ONE_a}_{vw} = \tilde{C}^{\ONE_a}_{vw},$ thus obtaining unitarily equivalent families of maps $\{\Phi\}$ and $\{\Psi\}$. Finally, with an application of Proposition \ref{op-eq} the proof is complete.
\end{proof}

In their paper (Corollary B, \cite{Coles2018TheAlgebras}), they classify right cyclic pivotal TLJ$(d)$ C$^*$-modules in terms of bipartite graphs equipped with a dimension function satisfying a Perron-Frobenius condition. (Compare with Remark \ref{MW condition}.) There is a clear indication that this result should generalize to the $*$-2-categorical context for unitary $\TLJ$-modules. We leave this exploration to a future work and limit ourselves to state the following conjecture:

\begin{conj}
Equivalence classes of right cyclic pivotal unitary $\TLJ$-modules correspond to MW-type bipartite balanced $\Gamma$-fair graphs.
\end{conj}

\bibliographystyle{alpha}

\clearpage

\end{document}